\documentclass[11pt,reqno]{amsart}

\usepackage{amsmath,amssymb,amsthm,amscd,amsfonts}
\usepackage{mathrsfs}
\usepackage{latexsym}
\usepackage{graphicx,multicol,color}

\usepackage{verbatim}

\usepackage[latin1]{inputenc}
\usepackage{textcomp}

\newtheorem{theorem}{Theorem}[section]
\newtheorem{lemma}[theorem]{Lemma}
\newtheorem{prop}[theorem]{Proposition}

\newtheorem{remark}[theorem]{Remark}
\newtheorem{definition}[theorem]{Definition}

\makeatletter
\@addtoreset{equation}{section}

\makeatother

\newcommand{\slim}{\mbox{\rm s-}\lim}
\newcommand{\w}[1]{\langle {#1} \rangle}

\newcommand{\ef}{ \nopagebreak \hfill $ \Box $ \vskip 3mm}

\newcommand{\be}{\begin{equation}}
\newcommand{\ee}{\end{equation}}
\newcommand{\bea}{\begin{eqnarray}}
\newcommand{\eea}{\end{eqnarray}}

\newcommand{\f}{\frac}

\newcommand{\bC}{{\mathbb C}}
\newcommand{\bR}{{\mathbb R}}
\newcommand{\bN}{{\mathbb N}}

\newcommand{\bS}{{\mathbb S}}

\newcommand{\bZ}{{\mathbb Z}}

\newcommand{\vH}{{\mathcal H}}

\newcommand{\vL}{{\mathcal L}}

\newcommand{\vN}{{\mathcal N}}

\newcommand{\ran}{\mbox{\rm ran}}

 {\everymath{\displaystyle\everymath{}}\array}%
 {\endarray}

\begin{document}
\title[A Levinson's theorem]{A new Levinson's theorem for potentials with critical decay}

\author{Xiaoyao Jia}
\address[Xiaoyao Jia]{
 Department of Mathematics\\
Henan University of Sciences and Technology\\
471003 Luoyang, China}

\author{François NICOLEAU }
\address[François Nicoleau]{
Laboratoire de Mathématiques Jean Leray, UMR CNRS 6629\\
Universit\'e de Nantes\\
F-44322 Nantes Cedex 3, FRANCE}
\email{nicoleau@math.univ-nantes.fr}

\author{Xue Ping Wang}
\address[Xue Ping  Wang]{Laboratoire de Mathématiques Jean Leray, UMR CNRS 6629\\
Université de Nantes\\
44322 Nantes Cedex, France}
\email{xue-ping.wang@univ-nantes.fr}

\subjclass[2000]{35P25, 47A40, 81U10}
\keywords{Schr\"odinger operators, resolvent expansion, zero energy resonance,  spectral shift function,
Levinson's theorem }
\thanks{Research  supported in part by the French National Research Project NONAa, No. ANR-08-BLAN-0228-01}

\begin{abstract}
We study the low-energy asymptotics of the spectral shift function for Schr\"odinger operators
with potentials decaying like $O(\frac{1}{|x|^2})$. We prove a generalized Levinson's for this class of potentials in presence of zero eigenvalue and  zero
resonance.
\end{abstract}

\maketitle

\tableofcontents

\section{Introduction}

Threshold spectral analysis of Schr\"odinger operators plays an important role in many problems arising from non-relativistic quantum mechanics, such as low-energy scattering,  propagation of cold neutrons in a crystal or the Efimov effect in many-particle systems (cf. \cite{Bolle01, Newton02, Wang04}). Presence of zero resonance of Schr\"odinger operators is responsible for several striking physical phenomena.  The Levinson's theorem which relates the phase shift to the number of eigenvalues and the zero resonance  is one of the oldest topics in this domain (cf. \cite{Levinson, Ma01, Newton02}). Threshold spectral analysis is usually carried out for potentials decaying faster than $\frac{1}{|x|^2}$ (cf. \cite{JensenKato, Murata01, Wang03}).  Potentials with the critical decay  $\frac{1}{|x|^2}$ appear in many interesting situations such as spectral analysis on manifolds with conical end or ion-atom scattering for $N$-body Schr\"odinger operators. The threshold spectral properties for potentials with critical decay are quite different from those of more quickly decaying potentials and the contribution of zero resonance and zero eigenvalues to the asymptotics of the resolvent at low-energy and to the long-time expansion of wave functions are studied in \cite{Carron01, Wang02, Wang01}.
In this work, we shall study the low-energy asymptotics of some spectral shift function  and prove a generalized Levinson's theorem for  potentials with critical decay in presence of zero eigenvalue and zero resonance.
\\

Recall that for a pair of selfadjoint operators $(P, P_0)$ on some Hilbert space $\vH$, if
\begin{equation} \label{eq1.1}
 (P+ i)^{-k} -(P_0+ i)^{-k} \mbox{  is of trace class for some $k \in \bN^*$,}
\end{equation}
 the spectral shift function $\xi(\lambda)$ is defined as distribution on $\bR$ by
\begin{eqnarray}\label{fl004}
\text{Tr}~ (f(H) - f(H_0)) = - \int _{\mathbb R}
f'(\lambda)\xi(\lambda) ~d\lambda ,\quad \forall f \in \mathcal
S(\mathbb R).
\end{eqnarray}
In fact, this relation only defines  $\xi(\cdot)$ up to an additive constant. It can be fixed by, for example,  requiring $\xi(\lambda) =0$ for $\lambda $ sufficiently negative if $P$ and $P_0$ are both bounded from below.

In this paper, we are interested in the threshold spectral analysis of the  Schr\"{o}dinger operator $P = -\Delta
+v(x)$ on $L^2(\mathbb R^n)$ with $n \ge 2$ and $v(x)$  a  real  function satisfying
\begin{equation}\label{v}
v(x) = \frac{q(\theta)}{r^2} + O(\w{x}^{-\rho_0} ), \quad |x| >>1,
\end{equation}
for some  $q\in C (\bS^{n-1})$ and $\rho_0>2$. Here $x=
r\theta$ with $r=|x|$ and $\theta = \frac{x}{|x|}$. $\mathbb
S^{n-1}$ is the unit sphere . Let $-\Delta_{\bS^{n-1}}$ denote the Laplace-Beltrami operator  on
 $\mathbb{S}^{n-1}$. We assume throughout this paper that the
smallest eigenvalue $\lambda_1$ of $ -\Delta_{\bS^{n-1}}+q(\theta)$  verifies
\begin{eqnarray}\label{positive}
 \lambda_1> -\frac {1}{4} (n-2)^2.
\end{eqnarray}
This assumption ensures that the form associated to $\widetilde P_0 = -\Delta + \frac{q(\theta)}{r^2}$ on $C_0^\infty(\bR^n \setminus \{0\})$ is positive. We still denote by $\widetilde P_0$ its Friedrich's realization as selfadjoint operator on $L^2(\bR^n)$.
Notice that the condition (\ref{eq1.1}) is not satisfied  for the pair $(P, -\Delta)$ by lack of decay on $v$, nor for the pair $(P, \widetilde P_0)$ because of the critical singularity at $0$. This leads us to introduce a model operator $P_0$ in the following way. Let $0 \le \chi_j \le 1 ~( j=1,2)$ be smooth functions on $\mathbb
R^n$ such that supp $\chi_1 \subset B(0, R_1), \chi_1(x) = 1$ when
$|x| < R_0$ and
$$\chi_1(x)^2 + \chi_2(x)^2 = 1.$$
Set
\begin{eqnarray}
P_0 = \chi_1(-\Delta)\chi_1+\chi_2 \widetilde P_0 \chi_2 , \nonumber
\end{eqnarray}
on $L^2(\mathbb R^n)$. When $q(\theta)=0$, we can take $P_0 = -\Delta$ and the main results of this work still hold.
Under the  assumption (\ref{positive}), one has $P_0\ge 0$. The
operator $P$ will mainly be considered as a perturbation of
$P_0$. Under the condition that $v(x)$ is bounded and satisfies (\ref{v}) for some
 $\rho_0>n$, (\ref{eq1.1}) is satisfied for $k> \frac n 2$ and the spectral shift function $\xi(\lambda)$ for the pair $(P, P_0)$  is well  defined by (\ref{eq1.1}).
\\

High energy asymptotics of the spectral shift function of Schr\"{o}dinger operators has been
studied by many authors (see for example
\cite{Albeverio01},\cite{Pushnitski01},\cite{Robert01},\cite{Robert02},\cite{Yafaev01}).  In particular, D. Robert proved in  \cite{Robert01} that for the pair  $(P, P_0)$ introduced as above, if the potential $v$ is smooth and satisfies
\begin{eqnarray}\label{eq1.2}
|\partial_x ^\alpha (v(x)- \frac{q(\theta)}{r^2})| \le C_\alpha \langle x\rangle ^{-\rho_0 -|\alpha|}, \quad \mbox{ for $|x|$ large}.
\end{eqnarray}
for some $\rho_0 >n$, then one has

(i).  $\xi(\lambda)$ is $C^{\infty}$ in $(0,\infty)$;

(ii). $\frac{d^k}{d\lambda ^k}\xi(\lambda)$ has a complete
asymptotic expansion for $\lambda \rightarrow \infty$,
\begin{eqnarray*}
\frac{d^k}{d\lambda ^k}\xi(\lambda)\sim
\lambda^{n/2-k-1}\sum_{j\ge0}\alpha_j^{(k)}\lambda^{-j}.
\end{eqnarray*}
 We are  mainly interested in the low-energy asymptotics of  the spectral shift function $\xi(\cdot)$ for  $(P, P_0)$. Although $\xi(\cdot)$  depends on the choice of cut-offs $\chi_1$ and $\chi_2$, we shall see that physically interesting quantities in low-energy limit are independent of such choices.

The low-energy resolvent asymptotics for $P$ is studied in \cite{Wang02,Wang01} for $v$ of the form $v(x) =   \frac{q(\theta)}{r^2} + w(x)$ with $w(x)$ bounded and satisfying
$w(x) = O(\w{x}^{-\rho_0} )$ for some $\rho_0>2$.  In this work, in order to define the spectral shift function, we assume $v(x)$ to be bounded and in the last  section on Levinson's theorem we even need to assume it  smooth. Therefore, we need to slightly modify the proofs of \cite{Wang01} to suit with the present situation.

The eigenvalues of  the operator $-\Delta_{\bS^{n-1}} + q(\theta)$ play an important role in the threshold spectral analysis of $P$ (cf. \cite{Carron01,Wang02,Wang01}). Let
\begin{equation}
\sigma_{\infty}=\Bigl \{\nu; \nu=\sqrt{\lambda +\frac{(n-2)^2}{4}}, \lambda \in \sigma
(-\Delta_{\bS^{n-1}}+q(\theta))\Bigr \}.
\end{equation}
 Denote
\begin{equation}
\sigma_k=\sigma_\infty \cap ]0,k], \text{  } k \in \mathbb{N^*}.
\end{equation}
 $\sigma_1$  is closely related to the properties of zero resonance. We say that $0$ is a resonance of $P$ if the equation  $Pu=0$  admits a solution $u$ such that $\w{x}^{-1}u \in L^2$, but $u \not\in L^2$. $u$ is then called  a resonant state of $P$. For $\nu \in \sigma_{\infty}$, let $n_{\nu}$ denote the multiplicity of $\lambda_{\nu}=\nu^2-\frac{(n-2)^2}{4}$ as the eigenvalue of $-\Delta_{\bS^{n-1}}+q(\theta)$. For the class of potentials with critical decay, if zero is a resonance of $P$,  its multiplicity is at most
$\sum_{\nu \in \sigma_1} n_\nu$. Let $u$ be a resonant state of $P$. Then one has the following asymptotics for $r =|x|$ large
\begin{equation} \label{nuresonance}
u(x) =\frac{\psi(\theta)}{r^{\frac{n-2}{2} + \nu}} (1 + o(1))
\end{equation}
for some $\nu \in \sigma_1$ and $\psi \neq 0$ an eigenfunction of $-\Delta_{\bS^{n-1}}+q(\theta)$ with eigenvalue $\lambda_\nu$. We call $u$ a $\nu$-resonant state (or a $\nu$-bound state). This terminology is consistent with the historical half-bound state for rapidly decreasing potentials in three dimensional case, since the set $\sigma_1$ is then reduced to $\{\frac 1 2\}$. The multiplicity $m_\nu$  of $\nu$-resonant states is defined as the dimension of the subspace spanned by all $\psi $ such that the expansion (\ref{nuresonance}) holds for some resonant state $u$. The main result of this paper is the following generalized Levinson's theorem in presence of zero eigenvalue and zero resonance.

\begin{theorem} \label{th1.1} Let $n\ge 2$. Assume that $v$ is smooth and satisfies (\ref{eq1.2}) for some
$\rho_0> \max\{6, n+2\}$. Let $\xi(\cdot)$ denote the spectral shift function for the pair $(P, P_0)$. Then there exist some constants $c_j$ and $\beta_{n/2}$ such that
\begin{equation} \label{eq1.9}
\int_0^\infty (\xi'(\lambda) - \sum_{j=1}^{[\frac n 2]}  c_j \lambda ^{[\frac n 2]-j-1} ) d\lambda = - (\vN_- + \vN_0 + \sum_{\nu \in \sigma_1} \nu m_\nu) + \beta_{n/2}.
\end{equation}
Here $\beta_{n/2} =0$ if $n$ is odd and $c_{[\f n 2]} =0$ if $n$ is even. $\beta_j$ for $j=1,2$ is computed at the end of paper. $\vN_-$ is the number of negative eigenvalues of $P$, $\vN_0$  the multiplicity of the zero eigenvalue and $m_\nu$ that of $\nu$-resonance of $P$. It is understood  that if $0$ is not an eigenvalue (resp., if there is no $\nu$-resonant states), then $\vN_0=0$ (resp. $m_\nu =0$).
\end{theorem}

For the class of potentials under consideration, zero resonance of $P$ may appear in any space dimension with arbitrary multiplicity depending on $q$.  See \cite{Carron01,Wang02} in the case of manifolds with conical end and \cite{jia} in the case of  $\bR^n$. In \cite{jia}, the existence of zero resonance is studied. Levinson's theorem in presence of zero eigenvalue and zero resonance is mostly proved for spherically symmetric potentials by using Sturm-Liouville method or  techniques of Jost function. See \cite{Ma01,Newton02} for overviews on the topic, where the reader can also find  some discussions on  radial potentials with critical decay. The proof of Levinson's theorem for non-radial potentials is much more evolved. See \cite{Newton01} and a series of works of D. Bollé {\it et al.} \cite{Albeverio01,Bolle02,Bolle01} for potentials satisfying  $\w{x}^s v \in L^{1}(\bR^n)$ for some suitably  large $s$ depending on $n$ and for $n=1,2,3$.  For example, one needs   $s>4$ if $n=3$, $s>8$ if $n=2$ according to \cite{Bolle01}.  A formula as (\ref{eq1.9}) in its general setting seems to be new.

To prove Theorem \ref{th1.1}, we use the result of \cite{Robert01} on the high energy asymptotics of $\xi'(\lambda)$.
 The main task  of our work is to analyze the spectral shift function in a neighbourhood of $0$. Our approach is to calculate the trace of the operator-valued function $z \to (R(z) - R_0(z)) f(P)$  and its generalized residue at $0$, where $R_0(z) = (P_0-z)^{-1}$, $R(z) = (P-z)^{-1}$ and $f\in C_0^\infty(\bR)$ is equal to $1$ near $0$.  As the first step, we prove the asymptotics expansions of the resolvents $R(z)$ and $R_0(z)$ for $z$ near $0$ by adapting the methods of \cite{Wang01} to our situation. The main difficulties arise from the fact that the set $\sigma_\infty$ may be arbitrary (depending on $q(\theta)$). The interplays between the zero energy resonant states and between the zero resonant states and the zero energy eigenfunctions can  be produced  in  such  a way that their contributions may  be dominant over that of a single $\nu$-resonant state with itself.  This is overcome by carefully examining the  terms obtained from asymptotic expansion of the resolvent  and by  making use of the properties of $\nu$-resonant states.  Remark that the decay condition can be improved if $\sigma_1$ contains only one point and our proofs (both for  the resolvent expansions and for the Levinson's theorem) can be very much simplified if the potential $v(x)$ is spherically symmetric.

This work is organized as follows. In Section 2, we establish a
representation formula of the spectral shift function which will be used
to prove Levinson's theorem. In Section 3, we use the asymptotic
expansion of $\widetilde R_0(z) = (\widetilde P_0-z)^{-1}$ to get the asymptotic expansion
for the resolvents $R_0(z)$ and $R(z)$. The methods are the same as in \cite{Wang01}. We indicate only the modifications to make and omit the details of calculation. The hard part of the proof of Theorem \ref{th1.1} is to calculate the generalized residue at $0$ of the trace function $z \to Tr (R(z)-R_0(z))f(P)$, which is carried out in Section 4. Here $f$ is some smooth function with
compact support equal to $1$ near $0$. This result is used to study the low-energy
asymptotics of the derivative of the spectral shift function in
Section 5 and to prove the Levinson's theorem. \\[3mm]

{\noindent \bf Notation}.  The scalar product on
$L^2(\bR_+; r^{n-1} dr)$ and $L^2(\bR^n)$ is denoted by $\w{\cdot,
\cdot}$ and that on $L^2(\bS^{n-1})$  by $(\cdot, \cdot)$.
$H^{r,s}$, $r, s\in\bR$, denotes the weighted Sobolev
space of order $r$ with  the weight $\w{x}^{s}$. The duality
between $H^{1,s}$ and $H^{-1, -s}$ is identified with
$L^2$-product.  Denote $H^{0,s} = L^{2,s}$. $\vL({r,s};
{r',s'})$ stands for the space of continuous linear operators
from $H^{r,s}$ to $H^{r',s'}$.  The complex plane $\bC$ is slit along  positive real axis
so that $z^\nu = e^{\nu \ln z}$ and $\ln z = \log |z| + i\arg z$
with $0< \arg z < 2 \pi$ are holomorphic for $z$ near $0$ in the slit complex plane.

% -----------------------------------------------------------
%                              SECTION II
% ------------------------------------------------------------

\section {A representation formula}

Let $(P,P_0)$ be a pair of self-adjoint operators, semi-bounded from below, in some separable Hilbert space $\mathcal{H}$.
We assume that for some $k \in \mathbb N ^*$,
\be\label{ass1}
||(P-i)^{-k} -(P_0-i)^{-k}||_{tr} < \infty ,
\ee
where $|| . ||_{tr}$ denotes the trace-class norm in $\mathcal{H}$. Under the assumption (\ref{ass1}), it is well-known that for any $f \in {\mathcal S}(\bR)$, $f(P)-f(P_0)$ is of trace class.

For $z \in \bC$ in the resolvent set of $P_0$ and $P$, we denote by $R_0 (z) =(P_0 -z)^{-1}$,
(resp. $R(z) =(P -z)^{-1}$) the resolvent of $P_0$, (resp. of $P$). Writing
\be \label{decomp}
(R(z)-R_0 (z)) f(P) = [R(z)f(P)-R_0 (z)f(P_0)] -R_0(z) (f(P)-f(P_0)),
\ee
we see immediately that $(R(z)-R_0 (z)) f(P)$ is of trace class.

\vspace{0,5cm}
\noindent
The spectral shift function (SSF, in short) $\xi (\lambda) \in L^1_{loc}(\mathbb R)$ is defined up to a constant
as a Schwartz distribution on $\bR$ by the equation
\be \label{ssf}
\text{Tr}~ (f(P) - f(P_0)) = - \int _{\mathbb R} f'(\lambda)\xi(\lambda) ~d\lambda ,\quad \forall f \in \mathcal
S(\mathbb R).
\ee

\vspace{0,2cm}\noindent
The right hand side can be interpreted as $\langle \xi ',f\rangle$,
where $\xi'$ is the derivative of $\xi$ in the sense of the
distributions, and  $\langle \cdot , \cdot\rangle$ denotes the pairing between ${\mathcal S'}(\bR)$ and ${\mathcal S}(\bR)$.

%  ----------------------------------------------------------------
%                            ASSUMPTIONS
%------------------------------------------------------------------

\vspace{0,5cm}\noindent
In this section, we give a representation formula for the SSF  under the following assumptions :

\vspace{0,3cm}
\begin{itemize}
 \item The spectra of $P$ and $P_0$ are purely absolutely continuous in $]0, +\infty[$ :

\begin{eqnarray}\label{ass2}
\sigma(P_0) =\sigma_{ac}(P_0) &=& [0, +\infty[, \\
\sigma_{ac}(P)&=& [0, +\infty[.
\end{eqnarray}
In particular, we assume there are no embedded eigenvalues of $P$ and $P_0$ in $]0,+\infty[$.

\vspace{0,3cm} \item The total number, $\vN_-$, of negative eigenvalues of $P$ is finite.

\vspace{0,3cm}
\item For any $f \in  C_0^\infty(\mathbb R)$, there exists some $\epsilon_0>0$, $\delta_0 <1$ and $C >0$ (depending on $f$),
such that
\be\label{ass3}
|~\text{Tr} ~[(R(z)-R_0(z))f(P)]~| \le  \frac{C}{|z|^{1+\epsilon_0}},
\ee
for  $z\in \bC$ with $|z|$ large and $z\notin \sigma(P)$ and
\be\label{ass3bis}
|~\text{Tr} ~[R_0(z) (f(P)-f(P_0))]~| \le  \frac{C}{|z|^{\delta_0}},
\ee
for $z\in \bC$ with $|z|$ small and $z\notin \sigma(P_0)$.

\vspace{0,3cm}
\item  Let  $f\in C_0^\infty(\bR)$ with $f(t) =1$ for $t$ near $0$. The generalized residue of the function $z \to \text{Tr} ~[(R(z)-R_0(z))f(P)]$  at $z=0$ is finite  in the following sense:
if we denote, for $\epsilon \ll \delta$,
$$
\gamma (\delta, \epsilon)  =   \{z\in \bC \ ; |z| = \delta,
{\rm dist } (z, \bR^+) \ge \epsilon\},
$$
we assume that
\be \label{ass4}
J_0 = - \frac{1}{2\pi i} \lim_{\delta \rightarrow 0} \lim_{ \epsilon \rightarrow 0}
\int_{\gamma (\delta, \epsilon)} \text{Tr}~[(R(z)-R_0(z))f(P)] ~dz \ \ {\rm{exists}}, \ee where $\gamma (\delta,
\epsilon)$ is positively oriented.

\vspace{0,3cm} \item The derivative of $\xi$ in the sense of distributions satisfies
 \be \label{ass5}
  \xi'(\lambda) \in L^1_{loc}(]0,\infty[). \ee

\end{itemize}

Remark that conditions (\ref{ass1}), (\ref{ass2})-(\ref{ass3bis}) imply that if $J_0$ exists for some $f$ as above, then it exists for all  $f\in C_0^\infty(\bR)$ with $f(t) =1$ for $t$ near $0$ and $J_0$ is independent of $f$. In fact, if $f_1$ and $f_2$ are two such functions, applying (\ref{decomp}) to $f = f_1 -f_2$, one sees that
$\text{Tr} ~[(R(z)-R_0(z))f(P)]$ can be decomposed into two terms, one is holomorphic in $z$ near $0$ and the other is of the order $O(\frac{1}{|z|^{\delta_0}})$, $\delta_0<1$. Hence the generalized residue of $z \to \text{Tr} ~[(R(z)-R_0(z))f(P)]$ at $0$ is equal to zero. We have the following representation formula for the spectral shift
function $\xi(\lambda)$.

\begin{theorem}\label{rfth01}
Let $f\in \mathcal C_0^\infty (\mathbb R )$ such that $f(\lambda) = 1$ for
$\lambda$ in neighborhood of $\sigma_{pp}(P)\cup \{0\}$. Under the
above assumptions, the limit
\[ \int_0^\infty f(\lambda) \ \xi'(\lambda)\ d\lambda =\lim_{\delta  \to
0^+}\lim_{R\to\infty}  \int_\delta^R f(\lambda) \ \xi'(\lambda)\ d\lambda \]
exists and one has
\begin{eqnarray}\label{rffl02}
Tr \ (f(P)-f(P_0))-\int_{0}^\infty \xi'(\lambda) \ f(\lambda) \ d\lambda = \vN_-+ J_0.
\end{eqnarray}
\end{theorem}
\begin{proof} Let us consider the application $F(z) = Tr \ [(R(z) -R_0(z))f(P)]$ which is holomorphic
outside $\sigma (P)$.   We want to deduce (\ref{rffl02}) from Cauchy's formula applied to $F(z)$.

Let us begin by some notation. Let $E_j, \; 1\le j \le k,$ be the distinct
eigenvalues of $P$ with multiplicity $m_j$.  Denote for $z_0\in \bC$ and $\delta >0$,
\begin{equation}
C (z_0; \delta) = \{~z\in \mathbb C; ~|z-z_0| = \delta~\}\quad ;
\quad
D(z_0;\delta) =\{~z\in \mathbb C; ~|z- z_0| \le \delta~\}.
\end{equation}
For $\delta >0$ small enough, $\sigma(P)\cap D(E_i;\delta) = \{E_i\}$. For $R\gg1$ and $0<\epsilon \ll\delta$,
let us denote
\begin{eqnarray}
\gamma (R, \epsilon)  &=& \{z\in \bC; \ |z| = R,\  {\rm dist }(z, \bR^+) \ge \epsilon\} \\
%\gamma (\delta, \epsilon)  & =&    \{z\in \bC; \ |z| = \delta, \ {\rm dist} (z, \bR^+) \ge \epsilon\} \\
d^\pm (R, \delta, \epsilon)  &=&  [ \sqrt{\delta^2 - \epsilon ^2} \pm i \epsilon , \sqrt{R^2 - \epsilon ^2}
\pm i \epsilon ].
\end{eqnarray}

We denote by $\Gamma_{\delta, \epsilon,R}$ the  curve defined by
$$
\Gamma_{\delta, \epsilon,R} = \left(\bigcup_{j= 1}^k  C(E_j;\delta) \right) \cup \gamma (\delta, \epsilon)
\cup d^+(R, \delta, \epsilon) \cup \gamma (R, \epsilon) \cup d^- (R, \delta, \epsilon) .
$$
$\Gamma_{\delta, \epsilon,R}$ is positively oriented according to the anti-clockwise orientation of the big circle
$\gamma (R, \epsilon)$. Since $F(z)$ is holomorphic in the domain limited  by
$\Gamma_{\delta, \epsilon,R}$, the Cauchy integral formula  gives
$$
\nonumber \frac {1}{2\pi i}\oint_{\Gamma_{\delta,\epsilon,R}} F(z) \ d z  =  0.
$$
We split the integral into four terms
\begin{eqnarray}\label{rffl01}
\frac {1}{2i \pi }\oint _{\Gamma_{\delta,\epsilon,R}} F(z) d z &= &\sum_{j=1}^4 I_j \quad \mbox{ with }
\end{eqnarray}
\begin{eqnarray*}
I_1  =   \frac {1}{2i\pi } \oint _{\gamma(R, \epsilon)} F(z) ~ d z \ \ ,
& &
I_2 = \sum_{j= 1}^k  \frac {1}{2i\pi } \oint_{  C(E_j;\delta) } F(z) ~ d z .\\
I_3 = \frac {1}{2i\pi} \oint_{ \gamma (\delta, \epsilon)} F(z) ~ d z \ \ ,
& &
I_4  =  \frac {1}{2i\pi } \oint_{d^+ (R,\delta, \epsilon) \  \cup \ d^-(R,\delta,\epsilon)} F(z) ~ d z . \\
\end{eqnarray*}

 By condition (\ref{ass3}), one has $I_1 = O(R^{-\epsilon_0})$ when $R \to \infty$,
uniformly in $\delta$ and $\epsilon>0$. For $j=1, ...,k$, $E_j$ is an eigenvalue of $P$ with finite multiplicity $m_j$, and the spectral projection
associated to $E_j$ is given by
\begin{equation} \label{proj1}
\Pi_j = \frac {1}{2i\pi } \oint_{  C(E_j;\delta) } \ R(z) \ dz,
\end{equation}
since $C(E_j;\delta)$ is oriented according the clockwise orientation. Using that $R_0 (z)$ is holomorphic in
$D(E_j, \delta)$, we obtain :
\begin{equation} \label{proj2}
I_2 = \sum_{j= 1}^k \ Tr \ (\Pi_j f(P)) = \sum_{j= 1}^k   \ m_j f(E_j) = \sum_{j= 1}^k  \ m_j = \vN_-.
\end{equation}

 Using assumption (\ref{ass4}), we have by definition
\begin{equation} \label{residu}
\lim_{\delta \to 0} \ \lim_{\epsilon \to 0} I_3  = J_0,
\end{equation}
and as a conclusion, we have obtained :
\begin{equation} \label{calculI4}
\lim_{R \to +\infty}\ \lim_{\delta \to 0} \ \lim_{\epsilon \to 0}\  I_4  = -\vN_-- J_0.
\end{equation}

Now, let us establish a new expression for $I_4$; we split  $F(z)= F_1 (z) + F_2 (z)$ where
\begin{eqnarray}
F_1(z) &=& \text{ Tr} \ [R(z)f(P) - R_0(z) f(P_0)], \\
F_2(z) &= &\text{ Tr} \ [ R_0(z) (f(P_0)-f(P))].
\end{eqnarray}

First, let us investigate the contribution coming from $F_1 (z)$. By the definition of the SSF,
\begin{equation}
F_1(z) = -\int_\bR \xi (\lambda) \ { {\partial f} \over {\partial \lambda}} (\lambda,z) \ d\lambda
\end{equation}
where ${\displaystyle{ f(\lambda,z) = {1 \over {\lambda -z}} f(\lambda)}}$. Thus, $F_1 (z)= G_1 (z) -G_2 (z)$ with
\begin{eqnarray}
G_1(z) &=& \int_\bR \xi (\lambda) \  {1 \over {(\lambda -z)^2}} \  f(\lambda)\ d\lambda , \\
G_2(z) &=& \int_\bR \xi (\lambda)  \ {1 \over {\lambda -z}} \ f'(\lambda)\ d\lambda .
\end{eqnarray}
\noindent
For simplicity, we set $\Gamma = d^+ (R,\delta, \epsilon) \cup d^- (R,\delta, \epsilon)$.
We have
\begin{equation} \label{decompo}
\frac {1}{2i\pi } \oint_{\Gamma } \ F_1(z) \ dz = \frac {1}{2i\pi } \oint_{\Gamma } \ G_1(z) \ dz -
\frac {1}{2i\pi } \oint_{\Gamma } \ G_2(z) \ dz := (a) -(b).
\end{equation}

Let us study $(a)$. Using Fubini's theorem, we can write
\begin{equation}
(a) = \frac {1}{2i\pi } \int_\bR \xi (\lambda) \ f(\lambda)
\oint_{\Gamma } { {\partial} \over {\partial z}} \left({1 \over {\lambda -z}}\right) \ dz \ d\lambda.
\end{equation}
Set
\[
k(\lambda, s, \epsilon) =  \frac{1}{ \lambda - (\sqrt{s^2 -\epsilon^2} +i\epsilon)} -
\frac{1}{ \lambda - (\sqrt{s^2 -\epsilon^2} - i\epsilon)}.\]
One has
\begin{eqnarray}
(a)& = &\frac {1}{2i\pi } \int_\bR  \xi (\lambda) \ f(\lambda)
(  k(\lambda, R, \epsilon) -k(\lambda, \delta, \epsilon) ) \ d\lambda.
\end{eqnarray}
Let us look at the term $I_\epsilon$ related to $k(\lambda, R, \epsilon)$ and
\begin{comment}
\begin{equation}\label{Iepsilon}
I_{\epsilon} := \frac {1}{2i\pi } \int_\bR \xi (\lambda) \ f(\lambda) \
[\frac{1}{ \lambda - (\sqrt{R^2 -\epsilon^2} +i\epsilon)} -
\frac{1}{ \lambda - (\sqrt{R^2 -\epsilon^2} - i\epsilon)} ] \ d\lambda ,
\end{equation}
\end{comment}
let us show that ${\displaystyle{\lim_{\epsilon \to 0} I_{\epsilon} = \xi (R) f(R)}}$.
Fix $0<R_0 <R$ and
split $I_{\epsilon} = J_{\epsilon} + K_{\epsilon}$, where
\begin{equation}\label{Jepsilon}
J_{\epsilon} = \frac {1}{2i\pi } \int_{-\infty}^{R_0} \xi (\lambda) \ f(\lambda) \
k(\lambda, R, \epsilon)\ d\lambda, \quad  K_{\epsilon} = \frac {1}{2i\pi } \int_{R_0}^{+\infty} \xi (\lambda) \ f(\lambda) k(\lambda, R, \epsilon) \ d\lambda .
\end{equation}
By the basic properties of  the SSF,  one has $\xi f \in L^1 (\bR)$ and
\begin{equation}
| J_{\epsilon} | \leq \frac {1}{\pi } \int_{-\infty}^{R_0}  | \xi (\lambda) \ f(\lambda)|  \
\frac{\epsilon}{ (\lambda - \sqrt{R^2 -\epsilon^2})^2 +\epsilon^2} \ d\lambda  \leq C  \epsilon \ || \xi f||_1 .
\end{equation}
On the other hand, we have
\begin{equation}
K_{\epsilon} = \frac {1}{\pi } \int_{\frac{R_0 - \sqrt{R^2 -\epsilon^2}}{\epsilon}}^{+\infty} \
  (\xi f) ( \sqrt{R^2 -\epsilon^2} +\epsilon s) \ \frac{1}{s^2+1} \ ds .
\end{equation}
Under the assumption (\ref{ass5}), $\xi \in C^0 (]0,+\infty[)$, so by making use of the dominated
convergence theorem, one has
\begin{equation} \label{limitI}
\lim_{\epsilon \to 0}\  K_{\epsilon}  = \xi (R) f(R).
\end{equation}
As a conclusion, we have shown
\begin{equation} \label{limita}
\lim_{\epsilon \to 0}\  (a)  = \xi (R) f(R)-\xi (\delta) f(\delta) = -\xi (\delta),
\end{equation}
for $\delta >0$ sufficiently small and $R\gg1$.

 Now, let us study $(b)$. Using Fubini's theorem,
\begin{equation}
(b) = \int_\bR \xi (\lambda) \ f'(\lambda)
\left( \frac {1}{2i\pi } \oint_{\Gamma } {1 \over {\lambda -z}} \ dz \right)\ d\lambda.
\end{equation}
It is easy to check that ${\displaystyle{\frac {1}{2i\pi } \oint_{\Gamma } {1 \over {\lambda -z}} \ dz }}$
is uniformly bounded with respect to $\epsilon$ and
\begin{equation} \label{classical}
\lim_{\epsilon \to 0} \ \frac {1}{2i\pi } \oint_{\Gamma } {1 \over {\lambda -z}} \ dz =
{\bf 1}_{[\delta, R]} \ (\lambda) \ \ {\rm {a.e}} \ \ \lambda.
\end{equation}
So, making use of the dominated convergence theorem and then the assumption (\ref{ass5}), a straightforward application
of distribution theory gives
\begin{equation} \label{limitb}
\lim_{\epsilon \to 0} \ (b) = \int_{\delta}^R  \xi(\lambda) \ f'(\lambda) \ d\lambda =
- \xi (\delta)  -\int_{\delta}^R  \xi'(\lambda) \ f(\lambda) \ d\lambda,
\end{equation}
for $R \gg1$ and $\delta>0$ sufficiently small. As a conclusion, we have obtained,
\begin{equation} \label{limitF1}
\lim_{\epsilon \to 0} \  \frac {1}{2i\pi } \oint_{\Gamma } \ F_1(z) \ dz = \int_{\delta}^R \xi'(\lambda) \
f(\lambda) \ d\lambda.
\end{equation}

To study the contribution coming from $F_2 (z)$,  we write
\begin{equation}
\frac {1}{2i\pi } \oint_{\Gamma } \ F_2(z) \ dz = Tr \ \left( \frac {1}{2i\pi }
\oint_{\Gamma } \ R_0 (z) \ dz \ (f(P_0) -f(P)) \right).
\end{equation}
Using the spectral theorem for $P_0$ and the assumption
(\ref{ass1}), the same argument as (\ref{classical}) gives
\begin{equation}
\slim_{\epsilon \to 0} \ \frac {1}{2i\pi } \oint_{\Gamma } R_0 (z) \ dz \ = \ {\bf{1}}_{[\delta, R]} \ (P_0).
\end{equation}
Since $f(P)- f(P_0)$ is of trace class, and
\begin{equation}
\slim_{\delta \to 0, R\to\infty} \   {\bf{1}}_{[\delta, R]} \ (P_0) = Id,
\end{equation}
one can deduce (see Lemma \ref{rflm01} below) that
\begin{equation} \label{limitF2}
\lim_{R\to\infty}  \  \lim_{\delta \to 0}  \ \lim_{\epsilon \to 0} \ \frac {1}{2i\pi } \oint_{\Gamma } \ F_2(z) \
dz = Tr \ (f(P_0)-f(P)).
\end{equation}
So, using (\ref{calculI4}), (\ref{limitF1}) and (\ref{limitF2}), we obtain the result.
\end{proof}

To complete the proof of  Theorem \ref{rfth01}, we prove the following elementary lemma.

\vspace{0,2cm}
\begin{lemma}\label{rflm01}
Let $A$ be an operator of trace class. Assume that $f(\lambda)$ is an operator valued function, uniformly bounded
in $\lambda$, such that $\slim \limits_{\lambda \to  \lambda_0} f(\lambda) = B$ exists.
\par\noindent
Then, $f(\lambda)A$ converges to $BA$ in the trace norm class, as $\lambda\rightarrow \lambda_0$.
In particular,
\begin{equation}
\lim_{\lambda \to \lambda_0} Tr \ (f(\lambda)A)  =  Tr (BA).
\end{equation}
\end{lemma}
\noindent
{\bf{Proof.}} For any $\epsilon >0$, let $F$ be a finite rank operator such
that $||A-F||_{tr} <\epsilon$. Then,
$$
||f(\lambda)A -BA||_{tr} \le ||(f(\lambda) -B)F||_{tr} + ||(f(\lambda) -B)(A-F)||_{tr} .
$$
Since $F$ is a finite rank operator and $\slim \limits_{\lambda\rightarrow \lambda_0} f(\lambda) =B$,
we have $||(f(\lambda) -B)F||_{tr} \leq C \epsilon$
for $|\lambda-\lambda_0| \le \delta$ with some $\delta >0$ small enough. Since $f(\lambda) -B$ is a uniformly bounded in
$\lambda$, we have also $||(f(\lambda) -B)(A-F)||_{tr} \leq C \epsilon$. This ends the proof.
 \ef

The remaining part of this work is to apply Theorem \ref{rfth01} to
Schr\"{o}dinger operator, using the known results on the asymptotic
expansion of $\xi'(\lambda)$ as $\lambda\rightarrow \infty $. The
main task is to study  $\xi (\lambda)$ for $\lambda$ near $0$.
If one can calculate the generalized residue $J_0$ and can show that  $\xi'(\lambda)$ is integrable in $]0, 1]$,
then one can take a family of functions  $f_R(\lambda) =
\chi(\frac{\lambda}{R})$, where $\chi$ is smooth and $0 \le \chi(s)
\le 1, ~\chi(s) =1$ for $s$ near 0, $\chi(s) =0$ for $s>1$ and
expand  both the terms
\begin{eqnarray*}
\int_0^\infty  \xi'(\lambda) f_R(\lambda) ~d\lambda,  \text{ and }
~\text{Tr} (f_R(P)- f_R(P_0))
\end{eqnarray*}
in $R$ large. Theorem \ref{th1.1} can be derived from Theorem \ref{rfth01} by comparing the two asymptotic expansions in $R$.

%------------------------------------------------------------
%-----                     SECTION 3                     ----
%------------------------------------------------------------

\section{Resolvent asymptotics  near the threshold}

Consider the Schr\"odinger operator $P = -\Delta + v(x)$ where the potential $v(x)$ is bounded and  has the
asympotic behavior
\begin{equation}
v(x) = \frac{q(\theta)}{r^2} + O(\w{x}^{-\rho_0}), \quad |x|\to \infty
\end{equation}
where $\rho_0>2$,  $(r,\theta)$ is the polar coordinates on $\mathbb{R}^n$ and $q(\cdot)$ is continuous on the
sphere. Let $\widetilde P_0= -\Delta+ \frac{q(\theta)}{r^2} $ be the homogeneous part of $P$.
Assume $n \ge 2$ and
\begin{equation}\label{sigma}
 -\Delta_{\bS^{n-1}}+q(\theta) > -\frac {1}{
4} (n-2)^2, \quad \text {on } L^2(\mathbb{S}^{n-1}).
\end{equation}
 In particular,
(\ref{sigma}) implies that the quadratic form defined by $\widetilde P_0$ on $C_0^\infty(\bR^n \setminus \{0\} )$ is
positive. We still denote by $\widetilde P_0$ its Friedrich's realization as selfadjoint operator in $L^2(\bR^n)$.
 Due to troubles related to the critical local singularity in the definition  of spectral shift function, the model operator $P_0$ we
will use in the next Section  is a modification of $\widetilde P_0$ inside some compact set. Let $0 \le \chi_j \le 1
~( j=1,2)$ be smooth functions on $\mathbb R^n$ such that supp$\chi_1 \subset B(0, R_1), \chi_1(x) = 1$ when
$|x| < R_0$ and
$$\chi_1(x)^2 + \chi_2(x)^2 = 1.$$
Define
\begin{eqnarray}
P_0 = \chi_1(-\Delta)\chi_1+\chi_2 \widetilde P_0 \chi_2 , \nonumber
\end{eqnarray}
on $L^2(\mathbb R^n)$ (if $q=0$, we take $P_0= -\Delta$). In the next Section,  $P$ will be regarded as a perturbation of $P_0$. But to study the
low energy asymptotics of $R_0(z)$ and $R(z)$, we  regard both $P$ and $P_0$ as perturbations of $\widetilde P_0$.
Denote
\begin{eqnarray}
 P_0 & = & \widetilde P_0 -W, \\
 P & = & P_0+ V = \widetilde P_0 - \widetilde W
\end{eqnarray}
where
\begin{eqnarray*}
W &=&  \chi_1^2 \frac{q(\theta)}{r^2} - \sum_{j=1} ^2 |\nabla \chi_j|^2 \\
V&=&- \sum \limits_{i=1}^2 |\nabla  \chi_i|^2
+ v - \chi_2^2 \frac{q(\theta)}{r^2} \\
\widetilde W & = &  W -V = -v + \frac{q(\theta)}{r^2}.
\end{eqnarray*}
Note that $W$ and $\widetilde W$ contain  a critical  singularity at zero. For $n \ge 3$, the Hardy's
inequality implies that they map continuously $H^1$ to $H^{-1}$.  When $ n =2$, the same remains true under the condition (\ref{positive}) if one defines $H^1$ as the form domain of $\widetilde P_0$. $V$ is bounded and  satisfies
\begin{eqnarray}
| V(x)| \le C  \langle x\rangle^{-\rho_0}, \quad \rho_0>2.
\end{eqnarray}
 $P_0$ is still a Schr\"odinger operator with a potential of critical decay at the infinity.
But it has nice threshold spectral property at zero.

Let
\begin{equation}
\sigma_{\infty}=\Bigl \{\nu; \nu=\sqrt{\lambda +\frac{(n-2)^2}{4}}, \lambda \in \sigma
(-\Delta_{\bS^{n-1}}+q(\theta))\Bigr \}.
\end{equation}
 Denote
\begin{equation}
\sigma_k=\sigma_\infty \cap ]0,k], \text{  } k \in \mathbb{N}.
\end{equation}
 For $\nu \in \sigma_{\infty}$, let $n_{\nu}$ denote
the multiplicity of $\lambda_{\nu}=\nu^2-\frac{(n-2)^2}{4}$ as the eigenvalue of $-\Delta_{\bS^{n-1}}+q(\theta)$. Let
${\varphi_{\nu}^{(j)}, \nu \in \sigma_{\infty}, 1\leq j\leq n_{\nu}}$ denote an orthonormal basis of
$L^2(\mathbb{S}^{n-1}) $ consisting of eigenfunctions of $
-\Delta_{\bS^{n-1}}+q(\theta)$:$$(-\Delta_{\bS^{n-1}}+q(\theta))\varphi_{\nu}^{(j)}= \lambda_{\nu} \varphi_{\nu}^{(j)}, \text{ }
(\varphi_{\nu}^{(i)},\varphi_{\nu}^{(j)})=\delta_{ij} .$$
Let $\pi_{\nu}$ denote the orthogonal projection in $L^2(\bR_+ \times\mathbb{S}^{n-1}; r^{n-1}dr d\theta) $ onto
the subspace spanned by the eigenfunctions of $-\Delta_{\bS^{n-1}}+q(\theta)$ associated with the eigenvalue
$\lambda_{\nu}$ :
$$\pi_{\nu} f =\sum_{j=1}^{n_{\nu}} (f,\varphi_{\nu}^{(j)})\otimes  \varphi_{\nu}^{(j)},
\text{  } f\in L^2(\bR_+ \times\mathbb{S}^{n-1}; r^{n-1}dr d\theta). $$
 Let
$$Q_{\nu} = -\frac{d^2}{dr^2} - \frac{n-1}{r} \frac{d}{dr} +
\frac{\nu^2-\frac{(n-2)^2}{4}}{r^2}, \quad \text{ in } L^2(\mathbb{R}_+; r^{n-1}dr).$$

The asymptotic expansion of the resolvent  $R_0(z) = (P_0-z)^{-1}$ near zero can be obtained as in
\cite{Wang01}, by considering $P_0$ as a perturbation of $\widetilde P_0$. For the later purpose, we also need the
asymptotic expansion of $\frac{d}{dz}\widetilde R_0(z)$ for $z$ near 0, where
 $\widetilde R_0(z)=(\widetilde P_0-z)^{-1}$ for $z\notin \sigma(P_0)$.

Decomposition the resolvent $\widetilde R_0(z)$ as
$$\widetilde R_0(z) = \sum_{\nu \in \sigma_{\infty}}(Q_{\nu}-z)^{-1} \pi _{\nu}, \text{  } z \notin \mathbb R $$
  The Schwartz kernel $K_\nu(r, \tau; z)$ of
$(Q_{\nu} - z)^{-1}$, $\Im z >0$, can be calculated explicitly.
  In fact, the Schwartz kernel of $e^{-it Q_\nu}$ is given by (see \cite{Tay2})
\begin{equation}
  \frac{1}{2it} (r\tau )^{ -\frac{n-2}{2}} e^{ -
\frac{r^2 + \tau^2}{4it}  -i \frac{\pi \nu}{2}} J_\nu(
\frac{r\tau }{2t}), \quad t\in \bR,
\end{equation}
where $J_\nu(\cdot)$ is the Bessel function of the first
kind of order $\nu$. Since
\[
(Q_\nu -z)^{-1} = i \int_0^\infty e^{-it (Q_\nu-z)} \; dt
\]
for $\Im z>0$,  the Schwartz kernel of $(Q_\nu -z)^{-1}$ is
\begin{eqnarray}\label{Knu}
K_\nu(r, \tau  ; z) & = &  (r\tau )^{ -\frac{n-2}{2}}\int_0^\infty e^{ -
\frac{r^2 + \tau^2}{4it} + izt -i \frac{\pi \nu}{2}} J_\nu(
\frac{r\tau }{2t}) \;  \frac{dt}{2t} \\
& = &  (r\tau )^{ -\frac{n-2}{2}}\int_0^\infty e^{
\frac{i \rho }{t} + iz r \tau t -i \frac{\pi \nu}{2}} J_\nu(
\frac{1}{2t}) \;  \frac{dt}{2t} \nonumber
\end{eqnarray}
for $\Im z > 0$, where
$$\rho=\rho(r,\tau) \equiv \frac{r^2+\tau^2}{4r \tau}.$$
Note that the formula of $K_\nu(r, \tau  ; z)$ used in  \cite{Carron01, Wang02} contains a sign error.

The asymptotic expansion for $ (Q_\nu -z)^{-1} $ as $z\to 0$ and $\Im z>0$ is deduced from (\ref{Knu}) by splitting the last integral into two parts according to $t\in]0, 1$ or  $t\in[1, \infty[$. The integral for $t\in]0,1]$ gives rise to a formal power series expansion in $z$ near $0$. The expansion corresponding to the integral of $t\in[1, \infty[$ needs a lengthy calculation. See \cite{Wang02}. Set

$$f(s;r,\tau,\nu) = D_{\nu} (r,\tau) \int_{-1}^1 e^{i(\rho +\theta/2)s} (1-\theta^2)^{\nu-1/2}d\theta, \text{ } \nu \geq 0,$$
with
\begin{eqnarray}
D_{\nu}(r, \tau) &=& a_{\nu} (r\tau)^{- \frac{(n-2)}{2}}, \text{ } a_{\nu} = \frac{e^{-i \pi \nu /2}}{2^{2\nu +
1}\pi^{1/2} \Gamma (\nu +1/2)}.
\end{eqnarray}
Then $f$ can be expanded in a convergent power series in $s$
\begin{equation}
f(s;r,\tau,\nu) = \sum^{\infty}_{j=0} s^j f_j(r,\tau,\nu), s\in \mathbb{R},
\end{equation}
 with
\begin{equation} \label{fj}
f_j (r,\tau,\nu) = (r\tau)^{- \frac{1}{2}(n-2)} P_{j,\nu}(\rho),
\end{equation}
  with $P_{j,\nu}(\rho)$ a polynomial in $\rho$ of degree j:
$$P_{j,\nu}(\rho) = \frac{i^j a_{\nu}}{j!} \int_{-1}^1 (\rho + \frac{1}{2} \theta)^j
(1-\theta^2)^{\nu - \frac{1}{2}} d \theta.$$ In particular,
\begin{eqnarray}
f_0(r,\tau,\nu) &=& d_{\nu} (r \tau)^{-\frac{n-2}{2}},\text{ }d_{\nu} =  \frac{e^{-\frac{1}{2}i \pi \nu}}{2
^{2\nu + 1}
\Gamma(\nu + 1)}; \nonumber \\
 f_1(r,\tau,\nu) &=& i d_{\nu} (r
\tau)^{-\frac{n-2}{2}}\rho.\nonumber
\end{eqnarray}

 In \cite{Wang02, Wang01}, the expansion of $(\widetilde P_0 - z)^{-1}$  is given for $z$ near $0$. For the later use, we need the derivation of this expansion. By an abuse of notation, we denote by the same letter $F$  an operator on $L^2(\bR_+, r^{n-1} dr)$ and its distributional kernel.

For $\nu \in \sigma_{\infty}$, denote $[\nu]$ the integral part of $\nu$ and $\nu' = \nu -[\nu]$. Set
\begin{equation}
z_{\nu} = \left\{ \begin{aligned}
                   z^{\nu^{\prime}}, \text{ if } \nu \notin
                   \mathbb{N}, \\
                   z\ln z, \text{ if } \nu \in \mathbb{N}.
                   \end{aligned} \right. \nonumber
\end{equation}
For $\nu > 0$, let $[\nu]_-$ be the largest integer strictly less than $\nu$. When $\nu = 0$, set $[\nu]_- = 0$.
Define $\delta _{\nu}$ by $\delta_{\nu} =1$, if $\nu \in \sigma_{\infty}\cap \mathbb{N} $, $\delta_{\nu} = 0$,
otherwise. One has $[\nu] = [\nu]_- + \delta_{\nu}$.

\begin{prop}\label{lrapp01} Assume the condition (\ref{sigma}). One has

(a).  The following asymptotic expansion holds for $z$ near $0$ with $\Im z
>0$.
\begin{eqnarray} \label{res1}
\widetilde R_0(z) & = & \sum_{j=0}^{N} z^{j} F_j + \sum_{\nu \in \sigma_N}  \sum_{j=[\nu]_-}^{N-1} z_{\nu} z^j
G_{\nu,j+ \delta_\nu }\pi_\nu  + \widetilde R_0^{(N)}(z),
\end{eqnarray}
 in $ \mathcal L(-1,s; 1, -s), \;  s > 2N + 1$. The remainder term $R_0^{(N)}(z)$ can be estimated by
$$\widetilde R_0^{(N)}(z) = O (|z|^{N+\epsilon}) \in \mathcal {L} (-1,s;1,-s), s>2N+1,  $$
for some $\epsilon>0$ and $F_j$ is of the form
\begin{eqnarray}
F_j &= & \sum_{\nu \in \sigma_{\infty}} F_{\nu, j} \pi_\nu  \in \mathcal L(-1,s; 1, -s),\quad  s > 2j+1.
\end{eqnarray}
 $F_{\nu,j}$  and $G_{\nu, j}$ can be explicitly calculated from the asymptotic expansion of  $K_\nu(r, \tau, z)$ in $z$. In particular,
\begin{equation}
G_{\nu,j}(r,\tau)= \left\{ \begin{aligned}
         b_{\nu^{\prime},j}(r\tau)^{j+\nu^{\prime}} f_{j-[\nu]}( r,\tau ;\nu), ~\nu \notin \mathbb{N} \\
                  -\frac{(ir\tau)^j}{j!} f_{j-[\nu]}(r,\tau ;\nu),
                  ~\nu \in \mathbb{N}
                          \end{aligned} \right. \nonumber
\end{equation}
with $f_k$ defined by (\ref{fj}) and
\begin{eqnarray}\label{srfl02}
b_{\nu ^{\prime},j} = - \frac{i^j e^{-i\nu^{\prime} \pi/2}\Gamma(1 - \nu^{\prime})}
{\nu^{\prime}(\nu^{\prime}+1)\cdots(\nu^{\prime}+j)},
\end{eqnarray}
for $ 0<\nu^{\prime} < 1$.

(b). The expansion (\ref{res1}) can be differentiated with respect to $z$ and one has
\[
\frac{d}{dz}\widetilde R_{0}^{(N)}(z)  =  O( |z|^{N -1+\epsilon})  \in \mathcal L(-1,s;1, -s), \quad  s > 2N + 1,
\]
for some $\epsilon >0$ small enough.
\end{prop}

See \cite{Wang02} for the proof of (a). To show that it is possible to differentiate the asymptotic expansion in
$z$, one    first shows that it can be done for each  $(Q_\nu-z)^{-1}$, then one utilizes the properties of Bessel function to control the dependence on $\nu$, including the remainders, and finally takes the sum in $\nu$. See \cite{jia} for details.

\begin{remark}
For the later use, let us  precise  a few terms in $R_0(z)$. By an abuse of notation, we denote by same letter an operator on $L^2(\bR_+; r^{n-1} dr)$ and its distribution kernel. Then one has
\begin{eqnarray}
F_{\nu, 0} &=&  (r\tau)^{-\frac{1}{2}(n-2)}\int_0^\infty e^{i\frac{\rho}{t} -i\frac{\pi
\nu}{2}}J_{\nu}(\frac{1}{2t})~\frac{dt}{2t} , \quad \nu \in \sigma_\infty, \label{3.14} \\
G_{\nu,0} &= & -\frac{ e^{-i \pi \nu} \Gamma(1-\nu)}{\nu 2^{2\nu+1} \Gamma (1 + \nu)}
(r \tau)^{ - \frac{n-2}{2} + \nu}, \quad \nu \in ]0, 1[,  \label{Gnuzero} \\
G_{1,1} & = & -\frac{1}{8} (r \tau)^{ - \frac{n-2}{2} + 1}. \label{Gunun}
\end{eqnarray}
\begin{comment}
Note that the ``$-$" sign is missed for the coefficient of $G_{1,1}$  given in \cite{Wang02}, which leads to a sign error  related to $1$-resonant states in the resolvent expansion for $R(z)$.
\end{comment}
 By (\ref{3.14}), one can derive that
\begin{equation} \label{3.17}
|F_{\nu,0}(r, \tau)|\le C (r\tau)^{-\frac{n-2}{2}} (\frac{r\tau}{r^2 +\tau^2})^{ \min\{1, \nu\}},
\end{equation}
for some $C>0$ independent of $\nu \in\sigma_\infty$. The uniformity  in $\nu$ is obtained by  examining the dependence of $J_\nu(r)$ on $\nu$.
\end{remark}

The asymptotic expansion for $R_0(z) = (P_0-z)^{-1}$  can be deduced by
 $P_0$ as perturbations of $\widetilde P_0$. One has
\begin{equation} \label{reslventR0}
R_0(z) = ( 1- F(z))^{-1} \widetilde{R}_0(z), \quad  R(z) = ( 1- \widetilde F(z))^{-1} \widetilde{R}_0(z)
\end{equation}
where
\begin{equation}
F(z) =  \widetilde{R}_0(z)W, \quad  \widetilde F(z) =  \widetilde{R}_0(z) \widetilde W.
\end{equation}
For $n\ge 3$, the multiplication by $\frac{1}{|x|^2}$ belongs to $\vL(1, s; -1, s)$ for any $s$, by the Hardy inequality. By (\ref{Hardy}), the same is true for $n=2$ if we define $H^{1,s}$ as $\w{x}^{-s} Q(\widetilde P_0)$, where $Q(\widetilde P_0)$ is the form-domain of $\widetilde P_0$. Therefore although $W$ and $\widetilde W$ have a critical singularity $\frac{1}{|x|^2} $ at zero, Proposition \ref{lrapp01} implies that
 $F(z) = F_0W + O(|z|^{\epsilon})$ in $\vL(1,-s; 1,-s)$ for $s >1$ (here and in the following, $H^{1,s}$ is replaced by $\w{x}^{-s} Q(\widetilde P_0)$  when $n=2$). Similar result holds for $\widetilde{F}(z)$.
Note that $F_0W$ and $F_0 \widetilde W$ are not compact operators.

\begin{definition}  Set $\mathcal N (P) = \{~ u ;  F_0 \widetilde W u = u,~ u
\in H^{1,-s},~ \forall s
>1\}.$ A function $u \in \mathcal N (P) \backslash L^2$ is called a
resonant state of $P$ at zero. If $\mathcal N (P) =\{0\}$, we say that $0$ is the regular point of $P$. The
multiplicity of the zero resonance of $P$ is defined as $\mu_r = \dim \vN/(\ker_{L^2} P)$.  Zero resonance and resonant states of $P_0$ are defined in the same way with $\widetilde W$ replaced by $W$.
\end{definition}

For   $ u \in H^{1,-s}$ for any $s>1$ and $u \in \vN$, one can show that  $Pu= (\widetilde P_0 -\widetilde W)u=0$ If
$\widetilde W = O(\w{x}^{-\rho_0})$ with $\rho_0>3$, it is proved in  \cite{Wang02} that
\begin{equation} \label{rstates}
 u(r\theta) =\sum_{\nu \in \sigma_1}  \sum_{j = 1}^{n_\nu}
 \frac{1}{2\nu} \w{ \widetilde W u, |y|^{-\frac{n-2}{2} + \nu}\varphi_\nu^{(j)}}
 \frac{ \varphi_\nu^{(j)} (\theta)} { r^{\frac{n-2}{2} +\nu} } +
\widetilde u,
\end{equation}
where $ \widetilde u\in L^2(|x|>1)$ , and $(\cdot, \cdot)$ is the scalar
product in $L^2(\mathbb S^{n-1})$.  In particular,  (\ref{rstates}) shows that the multiplicity of the zero resonance of $P$ is bounded by the total multiplicity of eigenvalues $ \lambda_\nu$ of $ -\Delta_{\bS^{n-1}} + q(\theta)$ with $\nu \in \sigma_1$. \\

For $\nu \in \sigma_1$, we shall say $u $ is a {\it $\nu$-resonant state}, or  {\it $\nu$-bound state}, of $P$
if $u \in \vN(P)$ and if $u$ has an asymptotic behavior like
\[
 u(x) = \frac{\psi(\theta)}{r^{\frac{n-2}{2}+ \nu} } + o( \frac{1}{r^{\frac{n-2}{2}+ \nu}}),
\]
for some $\psi \neq 0$, as $r \to \infty$. In the case $n=3$ and $q(\theta)=0$, one has $\sigma_1=\{\frac 1
2\}$.  The only possible zero energy resonant states of $P$ are
half-bound states, which is in agreement with the usual terminology on this topic. In the general case, (\ref{rstates}) shows that $\psi$ is an eigenfunction  of $-\Delta_{\bS^{n-1}} +
q(\theta)$ associated with the eigenvalue $\lambda_\nu$. We shall say that $m$ $\nu$-resonant states of $P$,
denoted as $u_1, \cdots, u_m$, are linearly independent if
\[
 u_l(x) = \frac{\psi_l(\theta)}{r^{\frac{n-2}{2}+ \nu} }(1+  o(1)), \quad r\to\infty,
\]
with $\{\psi_1, \cdots, \psi_m\}$ linearly independent in $L^2(\bS^{n-1})$. Let $m_\nu$ denote the maximal
number of  linearly independent $\nu$-resonant states of $P$. Then, $m_\nu$ does not exceed the multiplicity of
the eigenvalue $\lambda_\nu$ of $-\Delta_{\bS^{n-1}} + q(\theta)$ and
\begin{equation}
\sum_{\nu\in \sigma_1} m_\nu = \mu_r.
\end{equation}
Note in particular that if $u$ is a $\nu$-resonant state, then one has
\begin{equation} \label{caract}
\w{ \widetilde W u, |y|^{-\frac{n-2}{2} + \mu}\varphi_\mu^{(j)}} =0
\end{equation}
for all $\mu \in \sigma_1$ with $\mu < \nu$ and for all $j$ with $ 1 \le j \le n_\mu$ and if $u$ is an eigenfunction of $P$ associated with the eigenvalue $0$, the above equality remains true for all $\nu \in \sigma_1$ and all  $j$ with $ 1 \le j \le n_\mu$ (see (\ref{rstates}) ).  These properties will be repeatedly
used in the calculation of the generalized residue in Section 4. Finally,  remark that if  $m_\nu \neq 0$ for some $\nu \in \sigma_1$,
we can choose  $m_\nu$ $\nu$-resonant states $u_l =  \frac{\psi_l(\theta)}{r^{\frac{n-2}{2}+ \nu} }(1+  o(1))$,
$1 \le l \le m_\nu$, such that $\{\psi_l\}$ is orthonormal in $L^2(\bS^{n-1})$.
Modifying the basis $\{\varphi_\nu^{(j)}\}$ used in the definition of the spectral projection $\pi_\nu$, we can assume without loss that $ \varphi_\nu^{(l)} = \psi_l$, $1 \le l \le m_\nu$. \\

The model operator $P_0$ to be used in the next Section has the following nice threshold spectral property.

\begin{lemma}\label{lem3.1}
Assume $n \ge 2$ and (\ref{sigma}).  Zero is a regular point of $ P_0$.
\end{lemma}
\begin{proof}
Recall that $P_0 = \widetilde P_0 - W$ with $W$ of compact support. Let $u\in \vN (P_0)$. Let $\nu_0
=\min \sigma_1 >0$, by (\ref{sigma}). Since $W$ is of compact support, $Wu \in H^{-1, t}$ for any $t>0$. It
follows from (\ref{rstates}) that $u \in H^{1, -s'}$ for any $s'>1-\nu_0$. On the other hand, $P_0u =0$ implies
that $-\Delta u = (W -\frac{q(\theta)}{|x|^2})u  \in L^{2, 2-s'}$. By the assumption (\ref{sigma}), one can
deduce that   there exists $c>0$ such that
\begin{equation}\label{bdd}
 \w{ |x|^{-2} \chi_2 u, \chi_2 u} \le c \w{ \widetilde P_0 \chi_2 u, \chi_2 u}.
\end{equation}
In fact, let $\rho \in C_0^\infty(\bR^n)$ with $\rho(x) =1$ for $|x| \le 1$. Set $u_m = \rho(x/m)\chi_2 u$, $m
\in \bN^*$.  Then $u_m \in H^1$, and by the ellipticity of $-\Delta$ one has in fact $u_m \in H^2$.
The assumption (\ref{sigma}) implies that there exists $\epsilon_0>0$ such that for $f \in H^1$ with compact
support in $\bR^n \setminus \{0\}$, one has
\[
\w{\widetilde P f, f} \ge \int \int (|\frac{\partial  f }{\partial r}|^2 + (\epsilon_0 - \frac{(n-2)^2}{4})
\frac{|f|^2}{r^2}) r^{n-1} dr d \theta.
\]
Making use of the Hardy inequality, one obtains for $n\ge 2$
\begin{equation} \label{Hardy}
 \w{ |x|^{-2} f,  f} \le \epsilon_0^{-1} \w{ \widetilde P_0  f,  f}.
\end{equation}
Since $u\in H^{1, -s}$ and $-\Delta u \in L^{2, 2-s}$  for any $s>1-\nu_0$,  $\nu_0>0$, we can take  $s \in
]1-\nu_0, 1[$. Applying (\ref{Hardy}) to $f= u_m$ and  taking the limit $m\to \infty$,  we  derive (\ref{bdd}) by noticing that the term related to $[-\Delta, \rho(x/m)]$ tends to $0$, due to the decay  of $u$.
(\ref{bdd}) implies in particular that $\w{\widetilde P_0 \chi_2 u, \chi_2 u} \ge 0$. The equation
\[
 \w{-\Delta \chi_1 u,  \chi_1 u} + \w{\widetilde P_0 \chi_2 u, \chi_2 u}  = \w{ P_0 u, u} = 0
\]
shows that  each term of the above sum vanishes. The estimate (\ref{bdd}) gives in turn  $\chi_2 u=0$. Now the
unique continuation theorem  shows that $u=0$.  This proves that zero is a regular point of $P_0$.
\end{proof}

The existence of the asymptotic expansion of the resolvent $R_0(z)$ can easily be obtained  by a method of perturbation.
Concretely, let $K(z)$ be defined by
\begin{equation}
K(z) =  (\chi_1 (-\Delta + 1 -z)^{-1}\chi_1 + \chi_2 (\widetilde P_0 -z)^{-1} \chi_2) (P_0-z) -1, \quad z\not\in\bR_+.
\end{equation}
$K(z)$ is  compact operator on $H^{1,-s}$, $s>1$. By  Proposition \ref{lrapp01}, the limit $K(0) =\lim_{z\to 0} K(z)$ exists  and is compact. The kernel of $1+ K(0)$ in $H^{1,-s}$ coincides with $\vN(P_0)$ (both are equal to  solutions to the equation $P_0u =0$, $u\in H^{1,-s}$)) which by Lemma \ref{lem3.1} is $\{0\}$. So $(1+ K(0))^{-1}$ is invertible in $H^{1,-s}$. From Proposition \ref{lrapp01},
we can derive the asymptotic expansion of $R_0(z)$ in suitable weighted spaces  from the formula
\begin{equation}
R_0(z) = (1 + K(z))^{-1} (\chi_1 (-\Delta + 1 -z)^{-1}\chi_1 + \chi_2 (\widetilde P_0 -z)^{-1} \chi_2).
\end{equation}
In particular,  $R_0(0) =\lim_{z\to 0, z\not\in \bR_+} R_0(z)$ exists in $\vL(-1, s; 1,-s)$, $s>1$ and
\begin{equation} \label{R00}
R_0(0) = (1 + K(0))^{-1} (\chi_1 (-\Delta + 1 )^{-1}\chi_1 + \chi_2 F_0 \chi_2).
\end{equation}
 However since $K(z)$ contains several terms induced by cut-offs, the expansion obtained in this way is too complicated to be useful in the proof of Levinson's theorem which requires detailed information on higher order terms. For this purpose, we use the resolvent equation $R_0(z) = (1 - \widetilde R_0(z) W)^{-1} \widetilde R _0(z)$ to obtain a more concise expansion.

\begin{prop} \label{prop3.4}  Assume (\ref{sigma}), $n\ge 2$ and $\rho_0>2$.

(a).  $1-F_0 W$ is  invertible  on  $H^{1,-s}$, $s>1$.

(b).  Let $s\in ]1, \rho_0/2[$.  $1-F_0 \widetilde W$ is a Fredholm operator on $H^{1,-s}$  with indices $(m,m)$, $m =\dim \vN(P) <\infty$. One has
$H^{1,-s} = \ker ( 1-F_0 \widetilde W)\oplus \ran (1-F_0 \widetilde W)$.
\end{prop}

\begin{proof}
(a). Lemma \ref{lem3.1} shows that $1-F_0 W$ is injective. Since $(1-F_0W)^* = 1-WF_0$, $F_0$ is injective from coker $ (1-F_0W)$ to $\ker (1-F_0W)$. Therefore coker $ (1-F_0W) =\{0\}$ and ran$(1-WF_0)$ is  dense in $H^{1,-s}$.
For any $u \in H^{1,-s}$, set
\[
f = u + R_0(0) (Wu)
\]
with $R_0(0)$ given by (\ref{R00}).
Then $f \in H^{1,-s}$. Since $F_0$ and $R_0(0)$ are limits of $\tilde R_0(z)$ and $R_0(z)$ in $\vL(-1,s; 1,-s)$, we can check that
\begin{equation}
(1-F_0W)f = u +  ( R_0(0) - F_0  -F_0 W R_0(0))Wu = u.
\end{equation}
This proves that $u \in $  ran$(1-F_0W)$ and  $1-F_0W$ is bijective on $H^{1,-s}$. The open mapping theorem shows that  $1-F_0W$ is invertible on $H^{1,-s}$.

To show (b), recall that $1+ R_0(0)V$ is a Fredholm operator with equal indices $(m, m)$, $m =\dim \ker (  1+ R_0(0)V)$, and that $H^{1,-s} = \ker (1+ R_0(0)V) \oplus \mbox{ \rm ran } (1+ R_0(0)V)$ (see \cite{Jensen02, Wang03}).   By the relation $P= P_0 + V = \tilde P_0 - \tilde W$, one can show that $ \vN(P) =\ker P =  \ker (1+ R_0(0)V)$  in $H^{1,-s}$, $s\in ]1, \rho_0/2[$ and that coker $(1- F_0 \widetilde W)$ is also of dimension $m$.
To show that the range of  $1- F_0 \widetilde W $ is closed, let $u\in H^{1,-s}$ and $\{u_n\} \subset \ran  (1- F_0 \widetilde W)$ such that
$u_n \to u$ in $H^{1,-s}$.  Let $v_n \in H^{1,-s}$ such that $u_n = (1- F_0 \widetilde W) v_n$.
Then $\widetilde P_0 u_n = Pv_n = (P_0 + V) v_n$. It follows that
\begin{equation} \label{add1}
(1+ R_0(0) W)u_n = (1+ R_0(0)V)v_n.
\end{equation}
Since $1+ R_0(0) W$ is continuous on $H^{1,-s}$,  the left-hand side of (\ref{add1}) converges to $(1+ R_0(0) W)u$ as $n \to \infty$, while the right-hand side is clearly in the range of $1+ R_0(0)V$. Since $R_0(0)V$ is compact, the range of $1+ R_0(0)V$ is closed. It follows from (\ref{add1}) that there exists $v\in H^{1,-s}$ such that $ (1+ R_0(0) W)u = (1+ R_0(0)V)v$. One can check that $u = (1- F_0 \widetilde W) v \in \ran (1- F_0 \widetilde W)$, which proves that the range of $(1- F_0 \widetilde W)$ is closed.  It follows  that
 $(1- F_0 \widetilde W)$ is a Fredholm operator with equal indices. The other affirmation of (b) can be proved in the same way as in \cite{Wang03}.
\end{proof}

Denote
$$\vec{\nu} = (\nu_1, \cdots, \nu_k)\in( \sigma_N)^k, \quad
z_{\vec{\nu}} = z_{\nu_1} \cdots z_{\nu_k},$$
$$\{\vec{\nu}\}  = \sum_{j=1}^k \nu'_j, \quad [\vec{\nu}]_-
=\sum_{j=1}^k [\nu_j]_-, \quad [\vec{\nu}] =\sum_{j=1}^k [\nu_j].$$ Here   $\nu'_j = \nu_j - [\nu_j]_-$ for $\nu_j>0$. From Propositions \ref{lrapp01} and \ref{prop3.4} (a) and the resolvent equation $R_0(z) = ( 1- \widetilde{R}_0(z) W)^{-1} \widetilde{R}_0(z)$, we obtain the following

\begin{prop}\label{lrapp03}
 The following asymptotic expansion holds for $z$ near $0$ with $\Im z
>0$.

(a). Let $N\in \bN$ and $s>2N+1$. Then there exists $N_0\in \bN$ depending on $N$ and $\min \sigma_\infty$ such that
\begin{eqnarray}\label{lrafl05}
R_0(z)  = \sum_{j=0}^N z^j R_j +\sum_{ \{\vec{\nu}\} + j\le N}^{(1)} z_{\nu} z^{j}
R_{\vec{\nu},j} + R_0^{(N)}(z),
\end{eqnarray}
 in $ \mathcal L(-1,s; 1, -s)$. Here the notation $\sum_{ \{\vec{\nu}\} + j\le N}^{(l)}
$ means the finite sum taken  over all $\{\vec{\nu}\}\in\sigma_N^k, k\ge l,$ and $j \ge [ \vec{\nu}]_-$ such that $ \{\vec{\nu}\} + j\le N$,
\begin{eqnarray}
R_0 &= &A F_0;  \quad \quad R_1= A F_1 A ^*;\\
 R_{\vec{\nu},0} &= & A G_{\nu_1,\delta_{\nu_1}} \pi_{\nu_1} W A G_{\nu_2,\delta_{\nu_2}} \pi_{\nu_2} W\cdots A
G_{\nu_k,\delta_{\nu_k}} \pi_{\nu_k} A^* \label{Rnuzero}
\end{eqnarray}
for $\vec {\nu} = (\nu_1,\nu_2,\cdots,\nu_k)$ with $A = (1- F_0 W )^{-1}$.  In particular, if $k=1$ and
$\vec{\nu} = \nu_1$, one has
\begin{equation} \label{Rnuzero1}
R_{\vec{\nu},0} = A G_{\nu_1,\delta_{\nu_1}} \pi_{\nu_1}A^*.
\end{equation}
$R_j $ (resp. $R_{\vec{\nu},j}$) are in $\mathcal {L}(-1,s;1,-s)$ for $s>2j+1$ (resp. for
$s>2j+\{\vec{\nu}\}+1$), and $R_0^{(N)}(z) =O(|z|^{N+\epsilon})$ in $\mathcal {L}(-1,s;1,-s)$, $s>2N+1$.

(b). The above expansion (\ref{lrafl05}) can be differentiated in $z$ and  one has the estimate
\begin{equation}\label{lrafl07}
\frac{d}{dz}R_0^{(N)}(z)  = O(|z|^{N-1+\epsilon}),
\end{equation}
 in $ \mathcal L(-1,s; 1, -s), \;  s > 2N + 1$, with some $\epsilon >0$.
\end{prop}

For the operator $P$,   zero may be an eigenvalue or a resonance of $P$. The  multiplicity of zero resonance may be large, but does not exceed the sum of  multiplicities  of  the eigenvalues $\lambda$  of $-\Delta_{\bS^{n-1}} + q(\theta)$ such that $\nu = \sqrt{\lambda + \frac{(n-2)^2}{4}} \in ]0,1]$. The existence of an asymptotic expansion of $R(z)$ can be obtained as in \cite{Wang01} by regarding $P$ as perturbation of $P_0$.  In order to obtain a more ``concise" expansion, we regard $P$ as perturbation of $\widetilde P_0$. Let
\begin{equation}
0< \varsigma_1 < \cdots < \varsigma_{\kappa_0} \le 1
\end{equation}
be the points in $\sigma_1$ such
 that $P$ has  $m_{\varsigma_j}$ linearly independent  $ \varsigma_j$-resonant states
 with $\sum_{j=1}^{\kappa_0} m_{\varsigma_j} = \mu_r$, $\mu_r$ being the multiplicity of zero resonance of $P$.  Modifying the basis of the eigenfunctions $\{\varphi_\nu^{(j)}\}$ of $-\Delta_{\bS^{n-1}} + q(\theta)$ if necessary, one can show (see (\ref{rstates}) and the remarks following it) that there exists a basis of
 $ \varsigma_j$-resonant states, $u_{j}^{(i)}$, $i =1, \cdots, m_{\varsigma_j}$ verifying
 \begin{equation}\label{fl12}
|c_{\varsigma_j}|^{1/2}\w{\widetilde{W} u_{j}^{(l)},   |x|^{-\frac{n-2}{2} + \varsigma_j } \varphi_{\varsigma_j}^{(l')}} =
\delta_{ll'}, \quad 1 \le  l \le m_{\varsigma_j}, 1 \le l' \le n_{\varsigma_j}, \quad 1 \le j \le {\kappa_0},
 \end{equation}
where $c_{\nu}$ is the coefficient of $G_{\nu, \delta_\nu}$ given by
\begin{equation} \label{f17}
c_\nu =  -\frac{e^{-i \pi \nu} \Gamma(1-\nu)}{\nu 2^{2\nu+1} \Gamma (1+ \nu )},  \nu \in ]0,1[; \quad c_1 =
- \frac 1 8.
\end{equation}
(see (\ref{Gnuzero}) and (\ref{Gunun})) and $\delta_{ll'}=1$ if $l =l'$; $0$ otherwise. As seen in \cite{Wang01}, we can choose a basis $\{\phi_j; j=1, \cdots, \mu\}$ of $\vN$, $\mu =\dim \vN$, such that
\[
\w{\phi_i, \widetilde W \phi_j} =\delta_{ij}.
\]
Without loss, we can assume that for $1 \le j \le \mu_r$, $\phi_j$ is resonant state of $P$, while for $ \mu_r + 1 \le j \le \mu$, $\phi_j$ is an eigenfunction of $P$. Define
\begin{eqnarray}
Q_r & = &\sum_{j=1}^{\mu_r} \w{\widetilde W \phi_j, \cdot} \phi_j \\
Q_e & = &\sum_{j=\mu_r + 1}^{\mu} \w{\widetilde W \phi_j, \cdot} \phi_j.
\end{eqnarray}

  The following result can be proved by studying an appropriate Grushin problem  for $(1- \widetilde R_0(z) \widetilde W)$ as in \cite{Wang01}. See \cite{jia} for the details.

\begin{theorem} \label{th1}
 Let $\mu =\dim \mathcal{N} \neq 0$ and $N \in \bN$. Assume
${\rho_0} > \max\{ 4N-2, 2N + 4\}.$
One has the following asymptotic expansion for $R(z)$ in
$\mathcal{L}(-1, s; 1, -s)$, $s > \max\{2N+1, 2\} $:
\begin{equation}\label{fl13}
R(z) =  \sum_{j=0}^{N-1}z^{j}T_{j} +   \sum^{(1)}_{\{\vec{\nu}\} +j \le N-1}z_{\vec{\nu}} z^j T_{{\vec{\nu}};
j}
  +
T_e (z)  + T_r(z) + T_{er}(z) +O(|z|^{N-1 + \epsilon})
\end{equation}
Here $T_j$ (resp., $T_{{\vec{\nu}},j}$) is in $\mathcal{L}(1,-s; -1, s)$ for $s>2j+1$ (resp.,  for $s
>2j+1+\{\vec{\nu}\}$),
\[
T_0 = \widetilde A F_0, \quad T_1= \widetilde  A F_1 (1+ \widetilde W \widetilde A F_0)
\]
with $\widetilde A = ( \Pi' (1 -  F_0 \widetilde W ) \Pi')^{-1}\Pi' $, $\Pi'$ is the projection from
$H^{1,-s}$ onto \mbox{ \rm ran} $( 1 -  F_0 \widetilde W )$ corresponding to the decomposition $H^{1,-s} = \ker ( 1 -  F_0 \widetilde W ) \oplus \ran ( 1 -  F_0 \widetilde W )$.
 The sum $\sum^{(1)}_{\{\vec{\nu}\} +j \le N}$ has the same meaning as in (\ref{lrafl05})
  and the first term in this sum  is $z_{\nu_0}$ with coefficient $T_{\nu_0,0}$ given by
$$ T_{\nu_0,0} = \widetilde A G_{\nu_0,\delta_{\nu_0}}\pi_{\nu_0} (1 + \widetilde W  \widetilde A F_0), $$
 where
$\nu_0$ is the smallest value of $\nu\in \sigma_{\infty}$. $T_{e}(z)$, $T_r(z)$ describe the contributions up to
the order $O(|z|^{N-1+ \epsilon})$ from eigenfunctions and resonant states, respectively, and $T_{er}(z)$ the
interaction between eigenfunctions and resonant states. One has
\begin{eqnarray*}
T_e(z) & = & -z^{-1} \Pi_0  +    \sum^{(-)}_{ j, \, \{\vec{\nu}\} +j \le N-1 }
z_{\vec{\nu}} z^j  T_{e;{\vec{\nu}};  j} \\
T_r(z) & =& \sum_{j=1}^{\kappa_0} z_{\varsigma_j}^{-1}( \Pi_{r,j} + \sum^{+, N}_{\alpha, \beta,\vec{\nu},l}
z_{\vec{\nu}} z^{|\beta|}(z_{\vec{\varsigma}})^{-\alpha -\beta} z^{l}
T_{r; \vec{\nu}, \alpha, \beta, l, j} ),  \quad \quad  \mbox{ with } \\
\Pi_{r,j} & = & e^{i \pi \varsigma'_j } \sum_{l=1}^{m_{\varsigma_j}}  \w{ \cdot, u_j^{(l)}} u_j^{(l)}, \quad j
=1,
 \cdots, {\kappa_0}, \\
T_{er}(z) &= &  \sum_{j=1}^{\kappa_0} z_{\varsigma_j}^{-1}( \Pi_0\widetilde W Q_eF_1 \widetilde W  \Pi_{r,j}  + \Pi_{r,j} \widetilde W Q_rF_1 \widetilde W \Pi_0  \\
& & + \sum^{+, N}_{\alpha, \beta,\vec{\nu},l} z_{\vec{\nu}}
z^{|\beta|}(z_{\vec{\varsigma}})^{-\alpha -\beta} z^{l} T_{er; \vec{\nu}, \alpha, \beta, l, j} ).
\end{eqnarray*}
 Here $ \varsigma_j' = \varsigma_j -[\varsigma_j]$, $\Pi_0$ is the spectral projection of $P$ at $0$, and
$T_e(z)$ is  of rank not exceeding Rank $\Pi_0$ with  leading singular parts  given by $\nu_j \in\sigma_2$:
\begin{eqnarray}
T_{e; {\vec{\nu}};  -1} &= &  (-1)^{k'+1} (\Pi_0 \widetilde W G_{\nu_1,1+\delta_{\nu_1}}\pi_{\nu_1} \widetilde W
)\cdots (\Pi_0 \widetilde W G_{\nu_{k'},1+ \delta_{\nu_{k'}}}\pi_{\nu_{k'}} \widetilde W )\Pi_0,
\end{eqnarray}
for ${\vec{\nu}} = (\nu_1, \cdots, \nu_{k'})\in \sigma_2^{k'}$ with $\{\vec{\nu}\}  \le 1$, $
(z_{\vec{\varsigma}})^{-\alpha} = (z_{\varsigma_1})^{-\alpha_1}\cdots
(z_{\varsigma_{\kappa_0}})^{-\alpha_{\kappa_0}}$.
The summation   $ \sum^{(-)}_{ j, \, \{\vec{\nu}\} +j \le N-1 }$ is taken over all indices
 $\vec{\nu} \in (\sigma_N)^k$, $k \in \bN^*$  and all $j\in \bZ$ with $j \ge [\vec{\nu}]_- -1$ such that
$  \{\vec{\nu}\} +j \le N-1  $, and the summation $\sum^{+, N}_{\alpha, \beta, \vec{\nu},l}$ is taken over all
possible $\alpha $, $\beta \in \mathbb{N}^{\kappa_0}$ with $1 \le |\alpha| \le N_0$, $|\beta| \ge 1$,
$\vec{\nu} =(\nu_1, \cdots, \nu_{k'})\in \sigma_N^{k'}$, $ k' \ge 2|\alpha| $,   for which there are at least
$\alpha_k$ values of $\nu_j$'s belonging to  $\sigma_1$ with $\nu_j \ge \varsigma_k$,  for $ 1 \le k \le
{\kappa_0}$, $l\in\mathbb{N}$, satisfying
$$|\beta| + \{\vec{\nu}\}  + l -  \sum_{k=1}^{\kappa_0} ( \alpha_k +
\beta_k) \varsigma_k \le N.$$
\end{theorem}

This theorem is proved \cite{Wang01} for the Schr\"{o}dinger operator of the form $P=\widetilde P_0 +V$ with $V$ bounded and satisfying $|V| \le C \langle x\rangle ^{-\rho_0}$. In the present work,  $V= -
\frac{\chi_1^2}{r^2}q(\theta) + \widetilde{W} $ has a critical singularity at $0$. So we can not directly use the result of \cite{Wang01}, but with the help of Proposition \ref{prop3.4}, one can follow the same line of proof to obtain Theorem \ref{th1} (see \cite{jia}). With the sign correction on $c_1$ and the formula (4.32) in \cite{Wang01}, one sees that a  $1$-resonant state $u_j^{(l)}$ gives rise to a singularity of the leading term $ \frac{1}{z\ln z}\w{\cdot, u_j^{(l)}}u_j^{(l)}$ instead of  $-\frac{1}{z\ln z}\w{\cdot, u_j^{(l)}}u_j^{(l)}$ as stated in Theorem 4.6 of \cite{Wang01}.

%------------------------------------------------------------
%-----              SECTION 4                      ----------
%------------------------------------------------------------

\section{Generalized residue of the trace function}

\vspace{0.3cm}

Let $f$ be a function satisfying the condition of Theorem \ref{rfth01}. As we have seen before,
$(R(z)-R_0(z))f(P)$ is of trace class for $z \notin \sigma (P)$. Moreover, the application $z \to T(z) =
\text{Tr}~ [(R(z) -R_0(z))f(P)]$ is meromorphic on $\mathbb C \backslash \mathbb R_+$. The goal of this section
is to calculate the generalized residue of $T(z)$ at $z=0$, in the sense of Section 2.

\vspace{0.5cm}
\par\noindent
First, let us recall some well-known results for the trace ideals $\mathscr{S}_p \subset {\mathcal
L}(L^2(\mathbb{R}^n))$, (see \cite[IX.4]{Reed04} for details).

\vspace{0.3cm}
\begin{definition}
For $1\leq p<\infty$, we say that a compact operator $A\in \mathscr{S}_p$ if $|A|^p  $ is a trace class
operator, where $| A| = \sqrt{A^* A}$, and we set $||A||_p =(\text{Tr}~|A|^p)^{1/p}$. For $p=\infty,\
\mathscr{S}_{\infty}$ is the set of the compact operators with $||A||_{\infty} = ||A||$.
\end{definition}

\par\noindent
We have the following properties : \vspace{0.3cm}

\begin{prop}(\cite[IX.4]{Reed04})\label{schatten}
Let $1\le p \le \infty$ and $p^{-1} +q^{-1} =1$.

(a) If $A\in \mathscr{S}_p $ and $B\in \mathscr{S}_q $, then $AB \in \mathscr{S}_1$ and $||AB||_1 \le ||A||_p
\cdot||B||_q$.

(b) $\mathscr{S}_p $ is a Banach space with norm $||\cdot||_p$.

(c) $\mathscr{S}_1 \subset \mathscr{S}_p  $ .

(d) If $A\in\mathscr{S}_p$, then $A^*\in\mathscr{S}_p$ and $||A^*||_p =||A||_p$.

\end{prop}

\vspace{0.5cm} \par\noindent By Proposition \ref{lrapp03}, for $z$ near 0 with $\Im z >0$, one has for $s>1$ :
\begin{equation}\label{estimR0}
R_0(z) = R_0 + R_0^{(0)}(z) \ \ {\rm{in}} \ \ \mathcal {L}(-1,s;1,-s),
\end{equation}
with $R_0^{(0)}(z) = O(\mid z\mid^{\epsilon})$. We deduce the following result :

\vspace{0.2cm}

\begin{lemma}\label{rtlm02}
For $m >n/2$, $s>3$ and  $z$ near 0, $\Im z > 0$, $\langle x\rangle^{-s} R_0(z)\langle x\rangle^{-s} \in
\mathscr{S}_m$ and there exists a constant C independent of $z$, such that $$||\langle x\rangle^{-s}
R_0(z)\langle x\rangle^{-s} ||_m \le C.$$
Moreover, $\langle x\rangle^{-s} (R(z) -R_0 )\langle x\rangle^{-s} = O(|z|^\epsilon)$  in $
\mathscr{S}_m$.
\end{lemma}

\begin{proof} Let $\chi \in C_0^\infty(\mathbb R)$ such that $\chi(r) = 1$ for $|r| <1$. Then $\langle x\rangle^{-s}
R_0(z)\langle x\rangle^{-s}$ can be written as
\begin{eqnarray*}
\langle x\rangle^{-s} R_0(z)\langle x\rangle^{-s} = F_1(z) +F_2(z),
\end{eqnarray*}
with
\begin{eqnarray*}
 F_1(z)=\langle x\rangle^{-s} R_0(z)\chi(P_0)\langle x\rangle^{-s}; \quad
 F_2(z)=\langle x\rangle^{-s} R_0(z)(1-\chi(P_0))\langle
 x\rangle^{-s}.
\end{eqnarray*}

First, let us study $F_2(z)$. Using the resolvent identity, we can  decompose $F_2(z)=F_{21}+F_{22}(z)$, where
\begin{eqnarray*}
 F_{21} &=& \langle x\rangle^{-s} R_0(-1)(1-\chi(P_0))\langle
 x\rangle^{-s},\\
 F_{22}(z)& =& (1+z)\ (\langle x\rangle^{-s} R_0(z)\langle x\rangle^{-s'})
 (\langle x\rangle^{s'} R_0(-1)(1-\chi(P_0))\langle x\rangle^{-s}).
\end{eqnarray*}

It is easy to check that $F_{21}$ and $\langle x\rangle^{s'} R_0(-1)(1-\chi(P_0))\langle x\rangle^{-s}$ are in
$\mathscr{S}_m$, if $s'>1$ is chosen close to 1. By (\ref{estimR0}), it is clear that  $\langle x\rangle^{-s}
R_0(z)\langle x\rangle^{-s'}$ is uniformly bounded, for $z$ near 0, $\Im z > 0$. Then we deduce that $F_2(z) \in
\mathscr{S}_m $, and $||F_2(z)||_m \le C$ with some constant $C$ independent of $z$.

Now, let us study $F_1 (z)$. We write $F_1(z) = \langle x\rangle^{-s} R_0(z)\langle x\rangle^{-s'} (\langle
x\rangle^{s'}\chi(P_0)\langle x\rangle^{-s}) $ with $s'>1$ close to $1$. Using a similar argument as above, we
can get that $||F_1(z)||_m \le C$ with some constant $C$ independent of $z$.

Therefore $\langle x\rangle^{-s} R_0(z)\langle x\rangle^{-s} \in \mathscr{S}_m$, and for some $C$ independent of
$z$,
\begin{eqnarray*}
||\langle x\rangle^{-s} R_0(z)\langle x\rangle^{-s} ||_m \le ||F_1(z)||_m +||F_2(z)||_m \le C .
\end{eqnarray*}

Repeating the same arguments with $R_0(z)$ replaced by $R_0(z) -R_0$, we obtain that
$\langle x\rangle^{-s} (R(z) -R_0 )\langle x\rangle^{-s} = O(|z|^\epsilon)$  in $
\mathscr{S}_m$.
\end{proof}

We can now establish the main result of this section :

\begin{theorem}\label{rtth01}
Assume that $\rho> \max \{ 6,n+2\}$ and $f$ satisfies the condition of  Theorem \ref{rfth01}. Then the
generalized residue of $T(z)=\text{Tr}~[(R(z) -R_0(z))f(P)]$ at $z=0$ is given by
\begin{eqnarray}\label{rtfl05}
J_0 = \mathcal{N}_0 + \sum_{j=1}^{\kappa_0} \varsigma_j m_j ,
\end{eqnarray}
where $\mathcal{N}_0$ is the multiplicity of zero as the eigenvalue of $P$ and $m_j$ the multiplicity of
$\varsigma_j$-resonance of zero.
\end{theorem}

\begin{proof}
The proof is rather long and is divided in two steps.  We write $f \approx g$ if $f(z)-g(z) = O(|z|^{-1+\epsilon})$ for some $\epsilon >0$, with $z$ in a neighborhood of $0$ and $\Im z >0$.

We  fix $k\in \mathbb N$ with $k> \frac{n}{2} +1$ and we use the following
resolvent identities :
\begin{eqnarray*}\label{resolvent}
R(z)- R_0(z) &=& - R_0(z)V R(z) \\
             &=& \sum_{j=1}^{k-1} (-1)^j (R_0(z)V)^j R_0(z) + (-1)^k R(z) (VR_0(z))^k .
\end{eqnarray*}
Decompose $T(z) = -\text{Tr}~ [R_0(z) V f(P) R(z) ]$ as $T(z) = T_1(z) + T_2(z)$ where
\begin{eqnarray}
T_1(z) &=&- \text{Tr}~ [R_0(z) V  f(P)(R_0(z) + \sum_{j=1}^{k-1} (-1)^j (R_0(z) V)^j R_0(z))] .\\
T_2(z) &=& (-1)^{k+1} \text{Tr} ~[R_0(z) V  f(P) R(z) (VR_0(z) )^k)].
\end{eqnarray}

 First, let us study $T_1(z)$. Using the cyclicity of the trace, one has, for $s>1$,
close to $1$,
\begin{eqnarray*}
\text{Tr}~ [R_0(z)V f(P) R_0(z)] &=& \text{Tr}~ [R_0^2(z)V f(P)] \\
                                 &=& \text{Tr} [(\langle x\rangle ^{-s}\frac{d}{dz}R_0(z)
                                 \langle x\rangle^{-s})\ (\langle x\rangle ^{s}V f(P)\langle x\rangle ^{s})].
\end{eqnarray*}
Since $\rho>n+2, \ \langle x\rangle ^{s}V f(P)\langle x\rangle ^{s} \in \mathscr{S}_1$. So, by  Proposition
\ref{lrapp03}, we deduce that
\begin{eqnarray}
\text{Tr}~ [R_0(z)V f(P) R_0(z)] \approx 0.
\end{eqnarray}
Similarly, we can get that the other terms of $T_1(z)$ give the same contribution. Therefore, we have obtained :
\begin{equation}\label{Tun(z)}
T_1(z) \approx 0.
\end{equation}

Now, let us  study $T_2(z)$. Since $k>\frac{n}{2}+1> \frac{n}{2}-2$, we can  define
$T_3(z)$ for $z \notin \sigma (P)$ by  :
\begin{equation}
T_3(z) =(-1)^{k+1} \text{Tr} ~[R_0(z) V R(z) (VR_0(z) )^k)].
\end{equation}
Then, for $s>1$ close to $1$, we have as previously,
\begin{eqnarray*}
T_2(z) -T_3(z) &=&(-1)^{k} \text{Tr} ~[R_0(z) V  R(z) (1-f(P))(VR_0(z) )^k)] . \\
               &=& (-1)^{k} \text{Tr} ~[(\langle x\rangle^{-s} \frac{d}{dz}R_0(z)
                    \langle x\rangle^{-s}) \langle x\rangle^{s} V (R(z) (1-f(P))) \\
                & &   \hspace{1.6cm} (VR_0(z))^{k-1}V\langle x\rangle^{s}].
\end{eqnarray*}
Since $1-f(t)$ is equal to 0 for $t$ near 0, $R(z)(1-f(P))$ is uniformly bounded for $z$ in a neighborhood of
$0$. Moreover, since $\rho_0>6$ and $k > \frac{n}{2} +1$, Proposition \ref{schatten} (a) and Lemma \ref{rtlm02}
imply that, for $z$ near $0$ and $z \notin \mathbb{R}_+$,
\begin{equation}
\mid\mid (VR_0(z))^{k-1}\ V \langle x\rangle^{s}\mid \mid_1 = O(1).
\end{equation}
By Proposition \ref{lrapp03} (b), we conclude that, for some $\epsilon >0$ and $z $ near 0 with $\Im z>0$,
\begin{equation} \label{Ttrois(z)}
T_2(z) -T_3(z) \approx 0.
\end{equation}
The main task of the proof is to estimate $T_3(z)$ at $0$.

\vspace{0,5cm} \noindent $\bullet$ {\bf{Step 1}}: Assume that 0 is not an eigenvalue of $P$.

\vspace{0.3cm}\noindent
 Using the cyclicity of the trace, we remark that
\begin{equation}
T_3(z)= (- 1)^{k+1}\text{Tr}~ [(\frac{d}{dz}R_0(z))V R(z) V(R_0(z)V)^{k-1}].
\end{equation}
Let us introduce the following notation :
\begin{equation*}
U_1 = sgn (V) \ |V|^{\frac{1}{2}}\ , \  U_2 =  |V|^{\frac{1}{2}},
\end{equation*}
\begin{equation*}
S_1(z) = U_1\frac{d}{dz}R_0(z)U_2\ ,\  S_2(z) = U_1 R(z)U_2 \ ,\ S_3(z) =U_1 R_0(z) U_2.
\end{equation*}
Then,
\begin{equation}
T_3(z) = (-1)^{k+1} \text{Tr} ~[S_1(z)S_2(z)S_3^{k-1}(z)].
\end{equation}
As previously, by Lemma \ref{rtlm02}, we have for $k >\frac{n}{2}+1$, $\rho_0 >6$  and
$z$ near 0 with $\Im (z) >0$,
\begin{equation} \label{estimatetrace}
||S_3^{k-1}(z)||_1 = O(1).
\end{equation}
By Proposition \ref{lrapp03} (b), with $N=1$,
\begin{equation}
\frac{d}{dz}R_0(z) = R_1 + \sum_{ \{\vec{\nu}\} \le 1}^{(1)} \frac{d}{dz} z_{\vec \nu}R_{\vec \nu, 0}  +
O(|z|^{\epsilon}),
\end{equation}
in $\mathcal L (-1,s;1,-s)$, for $s>3$ and some $\epsilon>0$. Since $\rho_0>6$, we deduce that
\begin{equation}
S_1(z) = S_{11}(z) + S_{12}(z),
\end{equation}
with
\begin{equation}
S_{11}(z) =  U_1 R_1 U_2 +    \sum_{ \{\vec{\nu}\} \le 1}^{(1)} \frac{d}{dz} z_{\vec {\nu}} U_1 R_{\vec
{\nu},0}U_2
\end{equation}
and
\begin{equation}\label{sundeux}
S_{12}(z) = O(|z|^\epsilon) \quad \text{ in } {\mathcal{L}} (L^2(\mathbb R^n)).
\end{equation}

 In the same way, we can use Theorem \ref{th1}, with $N=1$ (and $\rho_0 >6)$ to decompose
$S_2(z)$. Note that $T_e (z)=0, \ T_{er}(z)=0$ since $0$ is not an eigenvalue of $P$. So we have,
\begin{equation}
S_2(z) = S_{21}(z) + S_{22}(z) + S_{23}(z)
\end{equation}
with
\begin{eqnarray}
S_{21}(z) &=&   \sum_{j=1}^{\kappa_0} z_{\varsigma_j}^{-1} U_1 \Pi_{r,j}U_2, \hspace{2cm} \\
S_{22}(z) &=&   \sum_{j=1}^{\kappa_0} \sum^{+, 1}_{\alpha, \beta,\vec{\nu},l}
               z_{\varsigma_j}^{-1}z_{\vec{\nu}} z^{|\beta|}(z_{\vec{\varsigma}})^{-\alpha -\beta} z^{l}
                  U_1 T_{r;\vec{\nu}, \alpha, \beta, l, j} U_2 , \\
S_{23}(z) &=& O(1) \quad \text{ in } {\mathcal{L}} (L^2(\mathbb R^n)). \label{sdeuxtrois}
\end{eqnarray}

 It follows from the previous discussion that
\begin{equation}
T_3(z) = T_{31}(z) +T_{32}(z)+T_{33}(z)
\end{equation}
where
\begin{eqnarray}
T_{31}(z) &=& (-1)^{k+1} \  \text{Tr} ~[S_{11}(z)S_{21}(z)S_3^{k-1}(z)], \label{ttroisun} \\
T_{32}(z) &=& (-1)^{k+1} \  \text{Tr} ~[S_{11}(z)S_{22}(z)S_3^{k-1}(z)],  \label{ttroisdeux}\\
T_{33}(z) &=& (-1)^{k+1} \ \text{Tr} ~[(S_{11}(z)S_{23}(z) +S_{12}(z)S_{21}(z) \nonumber \\
          & & \hspace{1,8cm}     +S_{12}(z)S_{22}(z)+S_{12}(z)S_{23}(z))S_3^{k-1}(z)].
\end{eqnarray}

\vspace{0.3cm}
\par\noindent
First, let us study $T_{33}(z)$. We have the following result :

\begin{lemma}\label{T33}
\begin{equation}
T_{33}(z) \approx 0.
\end{equation}
\end{lemma}

\begin{proof}
Clearly, we have for some $\epsilon >0$,
\begin{equation}\label{sunun}
S_{11}(z) = O (|z|^{-1+\epsilon}) \ \ , \ \ S_{21}(z) = O (|z \ln z|^{-1}) \ \ \   {\rm{in}} \ \ \ {\mathcal{L}}
(L^2(\mathbb R^n)).
\end{equation}
In the same way, we want to estimate $S_{22}(z)$. We note that the summation $\sum^{+, 1}_{\alpha,
\beta,\vec{\nu},l} $ is taken over all possible $\alpha,\beta \in \mathbb N^{\kappa_0}$ with $1 \le |\alpha| \le
N_0$, $|\beta|\ge 1$, $\vec{\nu} = (\nu_1 ,\cdots,\nu_{k'}) \in \sigma_2^{k'}$, $k' \ge 2|\alpha|$, for which
there are at least $\alpha_k$ values of $\nu_j$'s belonging to $\sigma_1$ with $\nu_j \ge \varsigma_k $, for $1
\le k \le \kappa_0$ and $l \in \mathbb N$, satisfying the condition
$$
|\beta| + \{\vec{\nu} \}+l -\sum \limits_{k=1}^{\kappa_0}(\alpha_k +\beta_k) \varsigma_k \le 1.
$$
It follows that
\begin{equation}\label{subtle}
z^{|\beta|}(z_{\vec{\varsigma}})^{-\beta} = O(1) \ \ \ ,\ \ \
z_{\vec{\nu}}(z_{\vec{\varsigma}})^{-\alpha}=O(|z|^\epsilon).
\end{equation}
We emphasize that in the second estimate of (\ref{subtle}), we have used $k' \ge 2|\alpha|$. Then, we have
obtained,
\begin{equation} \label{sdeuxdeux}
S_{22}(z) \approx 0 \ \ \ {\rm{in}} \ \ \    {\mathcal{L}} (L^2(\mathbb R^n)).
\end{equation}
Then, the lemma follows from (\ref{estimatetrace}), (\ref{sundeux}), (\ref{sdeuxtrois}), (\ref{sunun}) and
(\ref{sdeuxdeux}).
\end{proof}

The following elementary lemma will be useful to estimate $T_{31}(z)$ and $T_{32}(z)$.

\begin{lemma}\label{reduction}
Let  $u \in {\mathcal N}(P)$, $A = (1-F_0W)^{-1}$. Then
\begin{equation*}
A^* Vu = -{\widetilde W} u \ \ ,\ \ R_0 Vu = -u.
\end{equation*}
\end{lemma}

\begin{proof}
By definition, if $u \in {\mathcal N}(P)$, $F_0 \widetilde{W}u=u$. Thus, $Vu= (W-\widetilde{W})u
=-(1-WF_0)\widetilde{W}u$ which proves the first equality. Using the same argument, if $Pu=(P_0 +V)u=0$ then
Proposition \ref{lrapp03} implies that $R_0Vu=-u$.
\end{proof}

Now, let us study $T_{31}(z)$; we shall see  that $T_{31}(z)$ gives the leading term of the generalized residue.
We denote $\delta_{ij}$ the Kronecker symbol : $\delta_{ij} =1$ if $i=j$, $\delta_{ij}=0$ otherwise.

\begin{lemma}\label{T31}
\begin{equation}
T_{31}(z)  \approx - \frac{1}{z} \sum_{j=1}^{\kappa_0} \varsigma_j m_{\varsigma_j} -
\frac{\delta_{1,\varsigma_{\kappa_0}}} {z\ln z} \ \left( m_1 - \sum_{l=1}^{m_1} <R_1V u_{\kappa_0}^{(l)}, V
u_{\kappa_0}^{(l)}>  \right) .
\end{equation}
\end{lemma}

\begin{proof}
It follows from (\ref{ttroisun}) that
\begin{eqnarray}
(-1)^{k+1} \ T_{31}(z) = &\text{Tr}& ~[U_1 R_1 V  \sum_{j=1}^{\kappa_0} z_{\varsigma_j}^{-1}\Pi_{r,j}U_2
\ S_3^{k-1}(z)] \hspace{2cm} \nonumber \\
+ &\text{Tr}& ~[\sum_{ \{\vec{\nu}\} \le 1}^{(1)} \frac{d}{dz} z_{\vec {\nu}} U_1 R_{\vec {\nu},0} V
\sum_{j=1}^{\kappa_0} z_{\varsigma_j}^{-1}\Pi_{r,j}U_2  \ S_3^{k-1}(z)] \nonumber  \\
=&& (a)+(b). \nonumber
\end{eqnarray}

% ------------------ La cas ou \varsigma_{\kappa_0}<1 -------------------

First, let us study $(b)$. As we will see, the cases $\varsigma_{\kappa_0} <1$ and
$\varsigma_{\kappa_0} =1$ are slightly different.

\vspace{0.5cm}\noindent $\circ$ {\it{ Case 1 :}} Assume that $\varsigma_{\kappa_0} <1$.
\\
 It follows that $\varsigma_j' =\varsigma_j$, and for $j=1, ..., \kappa_0$,
$$
\Pi_{r,j}  =  e^{i \pi \varsigma_j} \sum_{l=1}^{m_{\varsigma_j}}  < \cdot \ , u_j^{(l)}> u_j^{(l)}.
$$
We deduce that
\begin{eqnarray*}
(b) &=&  \sum_{j=1}^{\kappa_0} \sum_{l=1}^{m_{\varsigma_j}} \sum_{ \{\vec{\nu}\} \le 1}^{(1)}
        \ e^{i\pi\varsigma_j }\  z_{\varsigma_j}^{-1}
        \frac{d}{dz}z_{\vec{\nu}} \ \text{Tr }~[U_1 R_{\vec {\nu},0} V
        \ \langle U_2S_3^{k-1}(z) \cdot \ , u_j^{(l)} \rangle  u_j^{(l)}] \\
    &=&  \sum_{j=1}^{\kappa_0} \sum_{l=1}^{m_{\varsigma_j}} \sum_{ \{\vec{\nu}\} \le 1}^{(1)}
        \ e^{i\pi\varsigma_j }\  z_{\varsigma_j}^{-1}
        \frac{d}{dz}z_{\vec{\nu}} \ \text{Tr }~[U_1 R_{\vec {\nu},0} V
        \ \langle \cdot \ , (S_3^{k-1}(z))^* U_2 u_j^{(l)} \rangle  u_j^{(l)}] \\
    &=&  \sum_{j=1}^{\kappa_0} \sum_{l=1}^{m_{\varsigma_j}} \sum_{ \{\vec{\nu}\} \le 1}^{(1)}
        \ e^{i\pi\varsigma_j }\  z_{\varsigma_j}^{-1}
        \frac{d}{dz}z_{\vec{\nu}} \ \langle U_1 R_{\vec {\nu},0} V u_j^{(l)},
        (S_3^{k-1}(z))^* U_2 u_j^{(l)} \rangle \\
    &=& \sum_{j=1}^{\kappa_0} \sum_{l=1}^{m_{\varsigma_j}} \sum_{ \{\vec{\nu}\} \le 1}^{(1)}
        \ e^{i\pi\varsigma_j }\  z_{\varsigma_j}^{-1}
        \frac{d}{dz}z_{\vec{\nu}} \  \langle  R_{\vec {\nu},0} V u_j^{(l)},
         U_1 (S_3^{k-1}(z))^* U_2 u_j^{(l)} \rangle.
\end{eqnarray*}

\vspace{0.1cm}\noindent Using (\ref{Rnuzero}) with $\vec {\nu} = (\nu_1,\dots,\nu_p)$, we have
\begin{equation}
R_{\vec{\nu},0} V u_j^{(l)} = A G_{\nu_1,\delta_{\nu_1}} \pi_{\nu_1} W A G_{\nu_2,\delta_{\nu_2}} \pi_{\nu_2} W
                              \cdots A G_{\nu_p,\delta_{\nu_p}} \pi_{\nu_p} A^* V u_j^{(l)}.
\end{equation}
 Thus, Lemma \ref{reduction} implies
\begin{equation}\label{rappel}
R_{\vec{\nu},0} V u_j^{(l)} = -A G_{\nu_1,\delta_{\nu_1}} \pi_{\nu_1} W A G_{\nu_2,\delta_{\nu_2}} \pi_{\nu_2} W
                              \cdots A G_{\nu_p,\delta_{\nu_p}} \pi_{\nu_p} \widetilde{W} u_j^{(l)}.
\end{equation}
 It is easy to see that for $\vec {\nu} = (\nu_1,\dots,\nu_p)$ with $\sum \limits_{i=1}^p
\nu_i >\varsigma_j$, one has $z_{\varsigma_j}^{-1}\frac{d}{dz}z_{\vec {\nu}} = O(|z|^{-1+\epsilon})$ for some
$\epsilon >0$. It follows from (\ref{estimatetrace}) that it suffices to evaluate $(b)$ for $\vec {\nu} =
(\nu_1,\dots,\nu_p)$ with $\sum \limits_{i=1}^p \nu_i \leq \varsigma_j$,

\vspace{0.3cm} $\diamond$ Assume first $\exists i \in \{1, \cdots,p\}$ such that $\nu_i = \varsigma_j$.

\vspace{0.3cm}\noindent One deduces that $p=1$, i.e $\vec {\nu}=\nu_1$. Moreover, since
$\varsigma_{\kappa_0}<1$, one has $\nu_1 =\varsigma_j <1$. Using (\ref{Rnuzero1}), we obtain
\begin{equation}
R_{\vec{\nu},0} V u_j^{(l)} = -A  G_{\varsigma_j,0}\  \pi_{\varsigma_j} \widetilde{W} u_j^{(l)}.
\end{equation}
Recalling that
\begin{equation*}
\pi_{\varsigma_j}  =\sum_{l'=1}^{n_{\varsigma_j}} (\cdot \ ,\varphi_{\varsigma_j}^{(l')})\otimes
\varphi_{\varsigma_j}^{(l')},
\end{equation*}
and using (\ref{Gnuzero}) and (\ref{f17}), we deduce
\begin{eqnarray}
G_{\varsigma_j,0}\  \pi_{\varsigma_j} \widetilde{W} u_j^{(l)} &=& \sum_{l'=1}^{n_{\varsigma_j}}
\int_0^{+\infty}\ G_{\varsigma_j,0}(r,\tau) \ (\widetilde{W} u_j^{(l)}  \ ,\varphi_{\varsigma_j}^{(l')})
 \varphi_{\varsigma_j}^{(l')}\ \tau^{n-1} \ d\tau \nonumber \\
&=& c_{\varsigma_j} \ r^{-\frac{n-2}{2} +\varsigma_j} \ \sum_{l'=1}^{n_{\varsigma_j}} \ \langle
\widetilde{W}u_j^{(l)}, \mid x\mid^{-\frac{n-2}{2} +\varsigma_j} \varphi_{\varsigma_j}^{(l')} \rangle
\varphi_{\varsigma_j}^{(l')}.
\end{eqnarray}
It follows from the normalization condition of resonant states given in (\ref{fl12}) that
\begin{equation}
G_{\varsigma_j,0}\  \pi_{\varsigma_j} \widetilde{W} u_j^{(l)} = - e^{-i\pi\varsigma_j} \ \mid c_{\varsigma_j}
\mid^{\frac{1}{2}} \ r^{-\frac{n-2}{2} +\varsigma_j}\ \varphi_{\varsigma_j}^{(l)}.
\end{equation}
Thus, in this case, we have
\begin{equation}\label{conclusion1}
R_{\vec{\nu},0} V u_j^{(l)} =  e^{-i\pi\varsigma_j} \ \mid c_{\varsigma_j} \mid^{\frac{1}{2}} \ A \left( \mid
x\mid^{-\frac{n-2}{2} +\varsigma_j}\ \varphi_{\varsigma_j}^{(l)} \right).
\end{equation}

\vspace{0.3cm} $\diamond$ Assume now that $\forall i \in \{1, \cdots,p\}$, $\nu_i < \varsigma_j$.

\vspace{0.3cm}\noindent In particular $\nu_p <1$  and using (\ref{caract}), we obtain
\begin{eqnarray}
G_{\nu_p,0} \pi_{\nu_p} \widetilde{W} u_j^{(l)} &=& G_{\nu_p,0}\ \sum_{l'=1}^{n_{\nu_p}}
(\widetilde{W} u_j^{(l)} ,\varphi_{\nu_p}^{(l')})\otimes \varphi_{\nu_p}^{(l')} \\
&=& c_{\nu_p} \ r^{-\frac{n-2}{2} +\nu_p} \ \sum_{l'=1}^{n_{\nu_p}} \ \langle \widetilde{W}u_j^{(l)}, \mid
x\mid^{-\frac{n-2}{2} +\nu_p} \varphi_{\nu_p}^{(l')} \rangle
\varphi_{\nu_p}^{(l')} \\
&=&0.
\end{eqnarray}
Thus, in this case, using (\ref{rappel}), we have
\begin{equation}\label{conclusion2}
R_{\vec{\nu},0} V u_j^{(l)} = 0.
\end{equation}

As a conclusion of (\ref{conclusion1}) and (\ref{conclusion2}), we obtain
\begin{equation}
(b) \approx \sum_{j=1}^{\kappa_0} \sum_{l=1}^{m_{\varsigma_j}}
        \ \mid c_{\varsigma_j }\mid^{\frac{1}{2}} \  z_{\varsigma_j}^{-1}
        \frac{d}{dz} (z_{\varsigma_j})\   \langle A ( \mid x\mid^{-\frac{n-2}{2} +\varsigma_j}
        \varphi_{\varsigma_j}^{(l)} ),
        U_1 (S_3^{k-1}(z))^* U_2 u_j^{(l)} \rangle.
\end{equation}

 By Proposition \ref{lrapp03}, $(S_3^{k-1}(z))^* = (U_2 R_0 U_1)^{k-1}
+O(|z|^{\epsilon})$ in $L^2(\mathbb R^n)$. Moreover, since $\varsigma_j <1, \  z_{\varsigma_j}^{-1} \frac{d}{dz}
(z_{\varsigma_j}) = \frac{\varsigma_j}{z}$. Thus,

\begin{eqnarray*}
(b) &\approx&  \frac{1}{z} \sum_{j=1}^{\kappa_0} \sum_{l=1}^{m_{\varsigma_j}}
        \ \mid c_{\varsigma_j }\mid^{\frac{1}{2}} \  \varsigma_j
         \ \langle A ( \mid x\mid^{-\frac{n-2}{2} +\varsigma_j}
        \varphi_{\varsigma_j}^{(l)}) ,
        U_1 (U_2 R_0 U_1)^{k-1} U_2 u_j^{(l)} \rangle \\
    &\approx&  \frac{1}{z} \sum_{j=1}^{\kappa_0} \sum_{l=1}^{m_{\varsigma_j}}
        \ \mid c_{\varsigma_j }\mid^{\frac{1}{2}} \  \varsigma_j
         \ \langle A ( \mid x\mid^{-\frac{n-2}{2} +\varsigma_j}
        \varphi_{\varsigma_j}^{(l)} ),
        V (R_0 V)^{k-1} u_j^{(l)} \rangle \\
    &\approx&  \frac{(-1)^{k-1}}{z} \sum_{j=1}^{\kappa_0} \sum_{l=1}^{m_{\varsigma_j}}
        \ \mid c_{\varsigma_j }\mid^{\frac{1}{2}} \  \varsigma_j
        \  \langle A ( \mid x\mid^{-\frac{n-2}{2} +\varsigma_j}
        \varphi_{\varsigma_j}^{(l)} ),
        V u_j^{(l)} \rangle \\
    &\approx&  \frac{(-1)^{k}}{z} \sum_{j=1}^{\kappa_0} \sum_{l=1}^{m_{\varsigma_j}}
        \ \mid c_{\varsigma_j }\mid^{\frac{1}{2}} \  \varsigma_j
         \ \langle  \mid x\mid^{-\frac{n-2}{2} +\varsigma_j}
        \varphi_{\varsigma_j}^{(l)} ,
         \widetilde{W} u_j^{(l)} \rangle ,
\end{eqnarray*}
where we have used Lemma \ref{reduction} in the two last equations. Using again the normalization condition of
resonant states given in (\ref{fl12}), we obtain :
\begin{equation}
(b) \approx \frac{(-1)^k}{z} \sum_{j=1}^{\kappa_0} \sum_{l=1}^{m_{\varsigma_j}}
        \  \varsigma_j \ = \ \frac{(-1)^k}{z} \sum_{j=1}^{\kappa_0} m_{\varsigma_j}\varsigma_j.
\end{equation}

% -------------- Le cas ou \varsigma_{\kappa_0} =1 ----------------

\vspace{0.5cm}\noindent $\circ$ {\it{ Case 2 :}} assume that $\varsigma_{\kappa_0} =1$.

\vspace{0.3cm}\noindent  In this case, $\varsigma_j'= \varsigma_j$ for $j=1, ..., \kappa_0-1$ and
$\varsigma_{\kappa_0}'=0$. So we can write,
\begin{equation}\label{be}
(b) =  \sum_{j=1}^{\kappa_0-1} \sum_{l=1}^{m_{\varsigma_j}} \sum_{ \{\vec{\nu}\} \le 1}^{(1)}
        \ e^{i\pi\varsigma_j }\  z_{\varsigma_j}^{-1}
        \frac{d}{dz}z_{\vec{\nu}} \ \langle  R_{\vec {\nu},0} V u_j^{(l)},
         U_1 (S_3^{k-1}(z))^* U_2 u_j^{(l)} \rangle +(b'),
\end{equation}
where we set
\begin{equation}
(b') =  \sum_{l=1}^{m_1} \sum_{ \{\vec{\nu}\} \le 1}^{(1)}
        \ \frac{1}{z\ln z} \frac{d}{dz}z_{\vec{\nu}} \ \langle  R_{\vec {\nu},0} V u_{\kappa_0}^{(l)},
         U_1 (S_3^{k-1}(z))^* U_2 u_{\kappa_0}^{(l)} \rangle.
\end{equation}
The first part of $(b)$ in (\ref{be}) can be calculated exactly as in the case 1. Thus, one has
\begin{equation}
(b) \approx \frac{(-1)^k}{z} \sum_{j=1}^{\kappa_0-1} m_{\varsigma_j}\varsigma_j +(b').
\end{equation}
Now, let us study $(b')$. Using the same approach, an easy calculus gives
\begin{equation}
(b') \approx (-1)^{k-1} \ \sum_{l=1}^{m_1} \frac{1}{z\ln z} \frac{d}{dz}(z\ln z) \langle R_{1,0} V
u_{\kappa_0}^{(l)}, V u_{\kappa_0}^{(l)} \rangle.
\end{equation}
Using (\ref{Rnuzero1}), we have
\begin{equation}
R_{1,0} V u_{\kappa_0}^{(l)} =A G_{1,1} \pi_1 A^* V u_{\kappa_0}^{(l)} = -A G_{1,1} \pi_1 \widetilde{W}
u_{\kappa_0}^{(l)}.
\end{equation}
Recalling that
\begin{equation*}
\pi_{1}  =\sum_{l'=1}^{n_{1}} (\cdot \ ,\varphi_{1}^{(l')})\otimes \varphi_{1}^{(l')},
\end{equation*}
we deduce
\begin{eqnarray*}
G_{1,1} \pi_1 \widetilde{W} u_{\kappa_0}^{(l)} &=&\sum_{l'=1}^{n_{1}} \int_0^{+\infty} G_{1,1}(r, \tau)
( \widetilde{W} u_{\kappa_0}^{(l)}\ ,\varphi_{1}^{(l')}) \varphi_{1}^{(l')} \ \tau^{n-1} d\tau \\
& =& c_1  \ r^{-\frac{n-2}{2} +1 } \sum_{l'=1}^{n_{1}} \ \langle \widetilde{W} u_{\kappa_0}^{(l)}\ , \mid
x\mid^{-\frac{n-2}{2} +1 } \varphi_{1}^{(l')}  \rangle \varphi_{1}^{(l')},
\end{eqnarray*}
where we have used (\ref{Gunun}) and (\ref{f17}) in the last equation. As previously, it follows from the
normalization condition of resonant states that
\begin{equation}
G_{1,1} \pi_1 \widetilde{W} u_{\kappa_0}^{(l)} =  -\mid c_1\mid^{\frac{1}{2}} r^{-\frac{n-2}{2} +1 }
\varphi_{1}^{(l)}.
\end{equation}
Thus,
\begin{equation}
R_{1,0} V u_{\kappa_0}^{(l)} = \mid c_1\mid^{\frac{1}{2}} A\left( \mid x\mid^{-\frac{n-2}{2} +1 }
\varphi_{1}^{(l)} \right).
\end{equation}
It follows that
\begin{eqnarray*}
(b') &\approx& (-1)^{k-1} \ \sum_{l=1}^{m_1} \frac{1}{z\ln z} \frac{d}{dz}(z\ln z) \ \mid c_1\mid^{\frac{1}{2}}
              \langle A\left( \mid x\mid^{-\frac{n-2}{2} +1 } \varphi_{1}^{(l)} \right)\ , V u_{\kappa_0}^{(l)} \rangle \\
     &\approx& (-1)^k \ \sum_{l=1}^{m_1} \frac{1}{z\ln z} \frac{d}{dz}(z\ln z) \ \mid c_1\mid^{\frac{1}{2}}
                \langle \mid x\mid^{-\frac{n-2}{2} +1 } \varphi_{1}^{(l)}\ , \widetilde{W} u_{\kappa_0}^{(l)} \rangle \\
     &\approx&  (-1)^k m_1 \left(\frac{1}{z} + \frac{1}{z\ln z} \right),
\end{eqnarray*}
where we have used again the normalization condition of resonant states. Thus, in this case, we have obtained
\begin{equation}
(b) \approx \frac{(-1)^k}{z} \sum_{j=1}^{\kappa_0-1} m_{\varsigma_j}\varsigma_j + (-1)^k m_1 \left(\frac{1}{z} +
\frac{1}{z\ln z} \right).
\end{equation}

 As a conclusion, we have proved in all cases
\begin{equation}
(b) \approx \frac{(-1)^k}{z} \sum_{j=1}^{\kappa_0} m_{\varsigma_j}\varsigma_j + \delta_{\varsigma_{\kappa_0}, 1}
\ (-1)^k \frac{m_1}{z\ln z}.
\end{equation}

It remains to study $(a)$. Using the same strategy as for $(b)$, we obtain easily
\begin{equation}
(a) \approx (-1)^{k-1} \sum_{j=1}^{k_0} \sum_{l=1}^{m_{\varsigma_j}} e^{i\pi\varsigma_j'} \ z_{\varsigma_j}^{-1}
\langle R_1 V u_j^{(l)} \ , V u_j^{(l)} \rangle
\end{equation}
If $\varsigma_{\kappa_0} <1$, it is clear that $(a)$ is negligible. If $\varsigma_{\kappa_0} =1$, we obtain
\begin{equation}
(a) \approx \frac{(-1)^{k-1}}{z\ln z} \sum_{l=1}^{m_1} \langle R_1 V u_{\kappa_0}^{(l)} \ , V u_{\kappa_0}^{(l)}
\rangle,
\end{equation}
and the lemma is proved.
\end{proof}

For $T_{32}(z)$, we have the following

\begin{lemma}\label{T3}
\begin{equation}
T_{32}(z) \approx 0.
\end{equation}
\end{lemma}

\begin{proof}
Using (\ref{sunun}) and (\ref{sdeuxdeux}), we see that it suffices to prove that the following term is
negligible :
\begin{equation*}
(c)= \sum_{j=1}^{\kappa_0} \sum_{ \{\vec{\nu}\} \le 1}^{(1)} \sum_{\alpha, \beta,\vec{\nu_1},l}^{+, 1}
z_{\varsigma_j}^{-1}z_{\vec{\nu}_1} z^{|\beta|}(z_{\vec{\varsigma}})^{-\alpha -\beta}
  z^{l} \frac{d}{dz}z_{\vec{\nu}}
  \text{Tr} ~[U_1 R_{\vec \nu ,0} V T_{r;\vec{\nu}_1, \alpha, \beta, l, j}U_2 S_3^{k-1}(z)].
\end{equation*}
First, let us remark that if ${\displaystyle{\sum_{i=1}^p \nu_i \geq \varsigma_j}}$, the same argument as
(\ref{subtle}) implies that
\begin{equation}
z_{\varsigma_j}^{-1}z_{\vec{\nu}_1} z^{|\beta|}(z_{\vec{\varsigma}})^{-\alpha -\beta}
  z^{l} \frac{d}{dz}z_{\vec{\nu}} \approx 0 .
\end{equation}

\vspace{0.2cm}\noindent So, it remains to estimate $(c)$ when $\vec{\nu} =(\nu_1, ..., \nu_p)$ satisfies
${\displaystyle{\sum_{i=1}^p \nu_i < \varsigma_j}}$. From the proof of Theorem \ref{th1}, we know that
$T_{r;\vec{\nu}_1, \alpha, \beta, l, j}$ is a linear combination of
\begin{equation}
\Pi_{r,j}B_{r;\vec{\nu}_1, \alpha, \beta, l, j}~\quad \text { and}~\quad  A_{r;\vec{\nu}_1, \alpha, \beta, l, j}
\Pi_{r,j} \widetilde W G_{\mu,\delta_\mu} \pi_\mu  B_{r;\vec{\nu}_1, \alpha, \beta, l, j},
\end{equation}
where $A_{r;\vec{\nu}_1, \alpha, \beta, l, j}$ is a bounded operator in $\mathcal L (1,-s;1,-s),~s>3$ and
$B_{r;\vec{\nu}_1, \alpha, \beta, l, j}$ is a bounded operator in $\mathcal L (-1,s;-1,s)$ for $s>3$.

\vspace{0.5cm}\noindent $\circ$ Let us study the contribution coming from  $\Pi_{r,j}B_{r;\vec{\nu}_1, \alpha,
\beta, l, j}$.

\vspace{0.3cm}\noindent To simplify the notation, we write $B= B_{r;\vec{\nu}_1, \alpha, \beta, l, j}$. In this
case, as previously, we have
\begin{equation}
\text{Tr} ~[U_1 R_{\vec \nu ,0} V \Pi_{r,j} B U_2 S_3^{k-1}(z)] = \sum_{l=1}^{m_{\varsigma_j}} e^{i\pi
\varsigma_j'} \langle R_{\vec \nu ,0} V u_j^{(l)} , U_1 (S_3^{k-1}(z))^* U_2 B^* u_j^{(l)} \rangle .
\end{equation}
Since $\nu_p < \varsigma_j$, as in (\ref{conclusion2}), we obtain $R_{\vec \nu ,0} V u_j^{(l)}=0$.

\vspace{0.5cm}\noindent $\circ$ Now, let us study the contribution coming from $ A_{r;\vec{\nu}_1, \alpha,
\beta, l, j} \Pi_{r,j} \widetilde W G_{\mu,\delta_\mu} \pi_\mu  B_{r;\vec{\nu}_1, \alpha, \beta, l, j}$.

\vspace{0.3cm}\noindent We remark that, if $\mu < \varsigma_j$, the same argument as before and (\ref{caract})
imply $\Pi_{r,j} \widetilde W G_{\mu,\delta_\mu} \pi_\mu=0$. Thus, it suffices to study the case $\mu \geq
\varsigma_j$. To do this, we refer to the proof of Theorem 4.6  in \cite{Wang01}. The term $G_{\mu,\delta_\mu}
\pi_\mu  B_{r;\vec{\nu}_1, \alpha, \beta, l, j}$ comes from the expansion of $L_1(z) \widetilde W
(D_0+D_1(z))\widetilde R_0(z)$ with $D_0 \in \vL(1,-s; 1,-s)$ for any $s>1$ and $D_0|_\vN =0$ and
\begin{eqnarray*}
L_1 (z) & = &  \sum_{\mu\in \sigma_1} z_\mu G_{\mu, \delta_\mu} \pi_\mu \\
D_1(z) & = &   z D_1 + \sum_{\mu\in \sigma_1} z_\mu D_{\mu, 0}
\end{eqnarray*}
for some  $D_1$ and $D_{\mu, 0}$ in $\vL(1,-s; 1, -s)$ with $s>3$, so the coefficient in front of
$G_{\mu,\delta_\mu} \pi_\mu $ $B_{r;\vec{\nu}_1, \alpha, \beta,l,j}$ is $O(z^\mu)$. In the same way, the term
$A_{r;\vec{\nu}_1, \alpha, \beta, l, j} \Pi_{r,j}\widetilde W$ comes from the the expansion of
\[
(1-(D_0 + D_1(z))L_1(z) \widetilde W)I_r(z) Q_r - \Pi_r(z)  \widetilde W Q_r,
\]
 where $I_r(z)$ is of the form
\[
I_r(z) = (1 +  O(|z|^\epsilon)) \Pi_r(z)  \widetilde W,  \quad \Pi_r(z) = \sum_{j=1}^{\kappa_0}
\frac{1}{z_{\varsigma_j}} \Pi_{r,j}
\]
with $\Pi_{r,j}$ defined in Theorem \ref{th1} (see p. 1931 of \cite{Wang01}),  so the coefficient in front of
$A_{r;\vec{\nu}_1, \alpha, \beta, l, j} \Pi_{r,j}\widetilde W$ is $O(|z|^{-\varsigma_j+\epsilon})$. It follows
that the coefficient of $T_{r;\vec{\nu}_1, \alpha, \beta, l, j}$ in Theorem \ref{th1} is $O(|z|^{\epsilon  + \mu
- \varsigma_j}) = O(|z|^{\epsilon})$. As a conclusion, $(c)= O(|z|^{-1+\epsilon})$ and the lemma is proved.
\end{proof}

% ----------- Le cas ou 0 est valeur propre --------------

\vspace{0.5cm}\noindent $\bullet$ {\bf{Step 2 : }} Assume that $0$  is  an eigenvalue of $P$.

\vspace{0.5cm}\noindent We follow the same strategy. $S_1(z), S_2(z),S_3(z)$, $S_{11}(z), S_{12}(z)$ are the
same as before. Using Theorem \ref{th1} with $N=1$, $S_2(z)$ can be decomposed as
\begin{equation}
S_2(z) = U_1 R(z) U_2 = \sum_{j=1}^5 S_{2j}(z),
\end{equation}
with
\begin{eqnarray*}
S_{21}(z) &=& U_1 T_r (z) U_2 \\
S_{22}(z) &=& - \frac{1}{z} \ U_1 \Pi_0 U_2 \\
S_{23}(z) &=& \sum_{ j ,\ \{\vec{\nu}\}+j  \le 0}^{(-)} z_{\vec{\nu}} z^{j} U_1 T_{e;{\vec{\nu}};  j}U_2 \\
S_{24}(z) &=& U_1 T_{er} (z) U_2 \\
S_{25}(z) &=& O(1).
\end{eqnarray*}

\vspace{0.2cm}\noindent Thus, it follows that
\begin{equation}
T_3 (z) = (-1)^{k+1} \text{Tr} ~[ S_1 (z) S_2 (z) S_3^{k-1}(z)] = \sum_{j=1}^5 T_{3j} (z),
\end{equation}
where for $j=1, ..., 4$,
\begin{equation}
T_{3j}(z) =  (-1)^{k+1} \text{Tr} ~[ S_{11} (z) S_{2j} (z) S_3^{k-1}(z)],
\end{equation}
and
\begin{equation}
T_{35}(z) = (-1)^{k+1} \text{Tr} ~[ (S_{11} (z) S_{25} (z) + S_{12}(z) S_2 (z) ) S_3^{k-1}(z)].
\end{equation}

\vspace{0.2cm}\noindent It follows from (\ref{estimatetrace}), (\ref{sundeux}) and (\ref{sunun}) that $T_{35}
(z) \approx 0$. Now, let us establish the following lemma which will be useful to estimate the other terms.

\begin{lemma}\label{utile}
For ${\vec \nu} = (\nu_1, ..., \nu_p) \in (\sigma_1)^p$, one has :
\begin{equation}
R_{\vec \nu ,0} V \Pi_0 =0 \ \ ,\ \ \Pi_0 V R_{\vec \nu ,0} =0.
\end{equation}
\end{lemma}

\begin{proof}
We only prove the first assertion since the other one is similar. Let $\Psi_j$, $j=1, ..., \vN_0$, be an orthonormal basis of  the eigenspace of $P$ with eigenvalue $0$. As previously, one has
\begin{equation}
R_{\vec{\nu},0} V \Phi_j = -A G_{\nu_1,\delta_{\nu_1}} \pi_{\nu_1} W A G_{\nu_2,\delta_{\nu_2}} \pi_{\nu_2} W
                              \cdots A G_{\nu_p,\delta_{\nu_p}} \pi_{\nu_p} \widetilde W \Phi_j,
\end{equation}
and
\begin{eqnarray*}
G_{\nu_p,\delta_{\nu_p}} \pi_{\nu_p} \widetilde W \Phi_j &=& G_{\nu_p,\delta_{\nu_p}} \sum_{l=1}^{n_{\nu_p}}
(\widetilde W \Phi_j , \varphi_{\nu_p}^{(l)} ) \otimes \varphi_{\nu_p}^{(l)} ,\\
&=& c_{\nu_p} \  \langle  \widetilde W  \Phi_j, \mid x\mid^{- \frac{n-2}{2} + \nu_p}
\varphi_{\nu_p}^{(l)} )  \rangle r^{- \frac{n-2}{2} + \nu_p} \otimes\varphi_{\nu_p}^{(l)} , \\
&=& 0,
\end{eqnarray*}
where we have used (\ref{rstates}) in the last equation, with $u = \Phi_j \in L^2 (\mathbb{R}^n )$.
\end{proof}

\vspace{0.3cm}\noindent First, let us study $T_{32}(z)$. We have the following result :

\begin{lemma}\label{t32}
\begin{equation}
T_{32}(z) \approx - \frac{\vN_0}{z}.
\end{equation}
\end{lemma}

\begin{proof}
We can decompose $T_{32}(z)$ as
\begin{equation}
T_{32}(z)= I_1(z)+I_2(z)
\end{equation}
with
\begin{eqnarray*}
I_1(z) &=& \frac{(-1)^k}{z} \sum_{ \{\vec {\nu}\} \le 1}^{(1)} \frac{d}{dz} z_{\vec {\nu}}
\text{Tr }~[U_1R_{\vec{\nu},0}V \Pi_0U_2 S_3^{k-1} (z)] , \\
I_2(z) &=& \frac{(-1)^k}{z} \text{Tr }~[U_1 R_1 V\Pi_0 U_2 S_3^{k-1} (z)].
\end{eqnarray*}

\vspace{0.2cm}\noindent By Lemma \ref{utile}, $I_1 (z)=0$. Now, let us study $I_2(z)$. As previously, we have
\begin{eqnarray*}
\text{Tr }~[U_1 R_1 V\Pi_0 U_2 S_3^{k-1} (z)] &=& \sum_{j=1}^{\vN_0} \langle R_1 V \Phi_j, U_1
(S_3^{k-1} (z))^* U_2 \Phi_j \rangle \\
                                              &=& (-1)^{k-1} \ \sum_{j=1}^{\vN_0} \langle R_1 V \Phi_j, V \Phi_j
                                              \rangle
                                              + O(\mid z\mid^{\epsilon}).
\end{eqnarray*}
Recall that $R_1 = A F_1 A^*$ and $A^* V\Phi_j = -\widetilde W \Phi_j$. Thus,
\begin{equation}
\text{Tr }~[U_1 R_1 V\Pi_0 U_2 S_3^{k-1} (z)] = \sum_{j=1}^{\vN_0} \langle F_1 \widetilde W \Phi_j, \widetilde W
\Phi_j \rangle + + O(\mid z\mid^{\epsilon}).
\end{equation}
The lemma comes from the following result (see \cite{JensenKato}, \cite{Wang03}) :
\begin{equation}
\langle F_1 \widetilde W \Phi_i, \widetilde W \Phi_j \rangle = \delta_{ij}.
\end{equation}
\end{proof}

\vspace{0.2cm}\noindent
Now, we prove :

\begin{lemma}\label{t33}
\begin{equation*}
T_{33}(z) \approx 0.
\end{equation*}
\end{lemma}

\begin{proof}
Since $S_{23}(z) \approx 0$, it is easy to see that
\begin{equation*}
T_{33}(z) \approx (-1)^{k+1} \sum_{ \{\vec {\nu}\} \le 1}^{(1)} \ \sum_{ j, \{\vec{\nu_1}\}+j  \le 0}^{(-)}
\frac{d}{dz} z_{\vec {\nu}} z_{\vec{\nu_1}} z^{j} \text{Tr }~[U_1R_{\vec{\nu},0}V T_{e;{\vec{\nu_1}};  j}U_2
S_3^{k-1} (z)]
\end{equation*}
We note that $T_{e;{\vec{\nu}}; j} = \Pi_0 A_{e;{\vec{\nu}};  j} $ with $A_{e;{\vec{\nu}};  j}$ be a bounded
operator in $\mathcal L (-1,s;-1,s)$, $s>3$. So the result comes from Lemma \ref{utile}.
\end{proof}

\vspace{0.2cm}\noindent
We have the following estimate  for $T_{32}(z)$ :

\begin{lemma}\label{t34}
\begin{equation*}
T_{34}(z) \approx  \frac{\delta_{\varsigma_{\kappa_0}, 1}}{z\ln z} \ \text{Tr }~[ R_1 V (\Pi_0 \widetilde W Q_eF_1
\widetilde W \Pi_{r,\kappa_0} + \Pi_{r,\kappa_0} \widetilde W Q_r F_1 \widetilde W \Pi_0) V].
\end{equation*}
\end{lemma}

\begin{proof}
We have
\begin{equation}
T_{34}(z) = (-1)^{k+1} \text{Tr }~[(U_1 R_1 U_2 + \sum_{ \{\vec {\nu}\} \le 1}^{(1)} \frac{d}{dz} z_{\vec {\nu}}
U_1R_{\vec{\nu},0} U_2) (U_1 T_{er} (z) U_2) S_3^{k-1} (z)].
\end{equation}
First, we assume that $\varsigma_{\kappa_0} <1$. In this case, $U_1 T_{er} (z) U_2 \approx 0$ and it
follows that
\begin{equation}
T_{34}(z) \approx J_1(z) +J_2(z)+J_3(z),
\end{equation}
where
\begin{eqnarray*}
J_1(z)&=& (-1)^{k+1} \sum_{j=1}^{\kappa_0} \sum_{\{\vec {\nu}\} \le 1}^{(1)} \frac{d}{dz} z_{\vec {\nu}}
z_{\varsigma_j}^{-1}
\text{Tr}~[U_1 R_{\vec {\nu},0}V   \Pi_0 \widetilde W Q_eF_1 \widetilde W \Pi_{r,j} U_2 S_3^{k-1} (z)].\\
J_2(z)&=& (-1)^{k+1} \sum_{j=1}^{\kappa_0} \sum_{\{\vec {\nu}\} \le 1}^{(1)} \frac{d}{dz} z_{\vec {\nu}}
z_{\varsigma_j}^{-1}
\text{Tr}~[U_1 R_{\vec {\nu},0}V \Pi_{r,j} \widetilde W Q_r F_1 \widetilde W \Pi_0 U_2 S_3^{k-1} (z)].\\
J_3(z)&=& (-1)^{k+1} \sum_{j=1}^{\kappa_0} \sum_{\{\vec {\nu}\} \le 1}^{(1)}  \sum^{+,1}_{\alpha,
\beta,\vec{\nu}_1,l} \frac{d}{dz} z_{\vec {\nu}} z_{\varsigma_j}^{-1} z_{\vec{\nu_1}} z^{|\beta|}
(z_{\vec{\varsigma}})^{-\alpha -\beta} z^{l} \\
&& \hspace{4cm} \text{Tr}~[U_1 R_{\vec {\nu},0}V T_{er; \vec{\nu}_1, \alpha, \beta, l, j}U_2  S_3^{k-1} (z)].
\end{eqnarray*}

 By Lemma \ref{utile}, $J_1 (z)=0$. Now, let us study $J_2(z)$. As previously, the
leading contribution is obtained when $\vec {\nu} = \varsigma_j$. So we have :
\begin{eqnarray*}
J_2 (z) &\approx & \frac{(-1)^{k+1}}{z} \sum_{j=1}^{\kappa_0} \varsigma_j
\text{Tr}~[U_1 R_{\varsigma_j,0}V \Pi_{r,j} \widetilde W Q_r F_1 \widetilde W \Pi_0 U_2 S_3^{k-1} (z)] \\
&\approx& \frac{(-1)^{k+1}}{z} \sum_{j=1}^{\kappa_0} \varsigma_j \text{Tr}~[V \Pi_{r,j} \widetilde W Q_r F_1
\widetilde W \Pi_0 (V R_0)^{k-1} V R_{\varsigma_j,0}],
\end{eqnarray*}
where we have used the cyclicity of the trace. Moreover, it is easy to see that Lemma \ref{reduction} implies
\begin{equation} \label{pizero}
\Pi_0 (V R_0)^{k-1} = (-1)^{k-1} \Pi_0 \ \ ,\ \ \Pi_{r,j} (V R_0)^{k-1} = (-1)^{k-1} \Pi_{r,j}.
\end{equation}
Using the first assertion of (\ref{pizero}) and Lemma \ref{utile}, we deduce
\begin{eqnarray*}
J_2 (z) &\approx& \frac{1} {z} \sum_{j=1}^{\kappa_0} \varsigma_j
\text{Tr}~[V \Pi_{r,j} \widetilde W Q_r F_1 \widetilde W \Pi_0  V R_{\varsigma_j,0}] \\
&\approx& 0.
\end{eqnarray*}
Finally, using the same arguments as in Lemma \ref{T3}, we can show that $J_3 (z) \approx 0$.

\vspace{0.5cm}\noindent Now, let assume that $\varsigma_{k_0} =1$. We follow the same strategy as in the case
$\varsigma_{\kappa_0} <1$. It is easy to see that, in this case,
\begin{eqnarray*}
T_{34}(z) &\approx&  \frac{(-1)^{k+1}}{z\ln z} \ \text{Tr }~[ U_1 R_1 V (\Pi_0 \widetilde W Q_eF_1 \widetilde W
\Pi_{r,\kappa_0}
+ \Pi_{r,\kappa_0} \widetilde W Q_r F_1 \widetilde W \Pi_0) U_2 S_3^{k-1} (z)] \\
&\approx& \frac{(-1)^{k+1}}{z\ln z} \ \text{Tr }~[  R_1 V (\Pi_0 \widetilde W Q_eF_1 \widetilde W \Pi_{r,\kappa_0} +
\Pi_{r,\kappa_0} \widetilde W Q_r F_1 \widetilde W \Pi_0) (VR_0)^{k-1} V].
\end{eqnarray*}
Then, the lemma comes from (\ref{pizero}).
\end{proof}

\vspace{0.3cm}\noindent
Finally, $T_{31}(z)$ has been computed in the case where $0$ is not an eigenvalue of $P$.

\vspace{0.5cm}\noindent {\it{End of the proof of Theorem \ref{rtth01}.}}

\vspace{0.2cm}\noindent
As a conclusion, it follows from the above discussion that in all the cases, one has
\begin{equation} \label{Tz}
T(z) \approx  - \frac{1}{z} \left( \sum_{j=1}^{\kappa_0} \varsigma_j m_{\varsigma_j} + \mathcal{N}_0 \right)
+ \frac {C}{z\ln z},
\end{equation}
for some constant $C$ and for all $z$ near $0$ with $\Im z > 0$.  Using the relation  $T(z) =
\overline{T({\overline z} )}$ for $\Im z<0$, we deduce that $C$ is real and  (\ref{Tz}) holds for  $z$ near $0$ and $\Im z \neq
0$. Recall that the generalized residue is defined as
\begin{equation*}
J_0 = - \frac{1}{2\pi i} \lim_{\delta \rightarrow 0} \lim_{ \epsilon \rightarrow 0} \int_{\gamma (\delta,
\epsilon)} T(z) ~dz
\end{equation*}
if the limit exists, where $\gamma (\delta,\epsilon)$ is positively oriented.  We obtain from (\ref{Tz}) that
\[
 - \frac{1}{2\pi i}  \int_{\gamma (\delta,
\epsilon)} T(z) ~dz = \sum_{j=1}^{\kappa_0} \varsigma_j m_{\varsigma_j} + \mathcal{N}_0  + O(\frac \epsilon \delta) + O(\frac{1}{|\log \delta|}),
\]
uniformly in $0 <\epsilon << \delta$. Taking first the limit $\epsilon \to 0$, then the limit $\delta \to 0$, one  derives $ J_0 = \sum_{j=1}^{\kappa_0} \varsigma_j m_{\varsigma_j} + \mathcal{N}_0 $. This proves
Theorem \ref{rtth01}.
\end{proof}

\section{Proof of Levinson's theorem}

In this section, we  shall use  Theorem \ref{rfth01} to prove a Levinson's theorem for potentials with critical
decay. First, we have to verify that the conditions (\ref{ass1}), and $(2.4)-(2.9)$ hold. In the last section,
we have seen that (\ref{ass4}) is satisfied if $\rho_0 > max(6, n+2)$.

\vspace{0.3cm}
\par\noindent

Under the assumption (\ref{positive}), the Hamiltonian $P_0$ is positive, and it is well known that $P$ and
$P_0$ have no embedded positive eigenvalues, and the spectrum of $P$ and $P_0$ is purely absolutely continuous
on $]0, +\infty[$.

\vspace{0.3cm}
\par\noindent
Using Theorem \ref{th1} with $N=0$, one has for $s>2$, $z$ small with $\Im z>0$,
\begin{equation} \label{nonaccumul}
\mid \mid \w{x}^{-s} R(z) \w{x}^{-s} \mid \mid \leq \frac{C}{\mid z\mid}
\end{equation}
We deduce that the negative eigenvalues of $P$ can not accumulate at $0$.
Indeed, let $\varphi$ be an eigenfunction of $P$ associated with an eigenvalue $\lambda \in [-\delta, 0[$,
$\delta$ small. It is well known that that $\varphi$ decays exponentially. So, using
(\ref{nonaccumul}) with $z = \lambda +i \epsilon$, we obtain
\begin{equation}
\frac{1}{\epsilon} \mid \mid \w{x}^{-s} \varphi \mid \mid \leq \frac{C}{\mid \lambda +i\epsilon \mid}
\mid \mid \w{x}^{s} \varphi \mid \mid,
\end{equation}
which gives the contradiction when $\epsilon \rightarrow 0$.

\vspace{0.3cm}
\par\noindent
At least, for $\rho_0>n$ and  $k > \frac{n}{2}$, it
is well known that (\ref{ass1}) holds, (see for example, \cite{Robert01}, Theorems 1.1 -1.2).

% -----------------------   lemma assumption estimates  ---------------

\vspace{0.3cm}
\par\noindent
Now, let us prove the following elementary lemma in order to verify (\ref{ass3}) :

\begin{lemma}\label{ltlm01}
Assume $\rho_0 >n$. Then, for every $f\in C_0^\infty(\mathbb R)$, there exists $C>0$, such that :
$$
|~\text{Tr} ~[(R(z)-R_0(z))f(P)]~| \le  \frac{C}{|z|^2},
$$
uniformly in $z\in \bC$ with $|z|$ large and $z\notin \sigma(P)$.
\end{lemma}

\begin{proof}
For $z\notin \sigma(P)$, using the resolvent identity, we have :
\begin{eqnarray*}
(R(z) - R_0(z)) f(P) &=& - R(z)V R_0(z)f(P) \\
                     &=& -R(z)V R(z)f(P) - R(z)V R_0(z)V R(z)f(P).
\end{eqnarray*}
Let $f_1 \in C_0^{\infty}$ such that $f_1 f =f$. Using the cyclicity of the trace, we obtain :
\begin{equation*}
\text{Tr} ~[(R(z)-R_0(z))f(P)] = (1) +(2),
\end{equation*}
with
\begin{eqnarray}
(1) &=&-\text{Tr}[ (R(z) f_1 (P)) \  (f_1(P) V)\  (R(z) f(P)) ], \\
(2) &=& - \text{Tr} [ (R(z) f_1(P)) \  (f_1(P) \w{x}^{-\frac{\rho_0}{2}}) \nonumber \\
        & & \ \ \  (\w{x}^{\frac{\rho_0}{2}} V R_0(z)V \w{x}^{\frac{\rho_0}{2}})\ (\w{x}^{-\frac{\rho_0}{2}} f_1(P))\  (R(z) f(P))].
\end{eqnarray}

\vspace{0.3cm}\noindent
For $|z|$ large, $\Im z \neq 0$, we have (\cite{Robert03}),
\begin{eqnarray*}
\mid \mid \w{x}^{\frac{\rho_0}{2}} V R_0(z) V \w{x}^{\frac{\rho_0}{2}} \mid \mid =O(|z|^{-1/2}).
\end{eqnarray*}
Moreover, by the spectral theorem,
$$
\mid \mid R(z) f(P) \mid \mid = \frac{1}{dist (z, Supp \ f)} = O (\frac{1}{\mid z\mid}).
$$
The lemma follows now from the fact that
$\w{x}^{-\frac{\rho_0}{2}} f_1(P)$ is a Hilbert-Schmidt operator and  $f_1(P) V$ is of trace class.
\end{proof}

\vspace{0.3cm}
\par\noindent
In the same way, using Theorem \ref{lrapp03} with $N=0$, one has for $z \notin \sigma (P_0), \ \mid z\mid$ small,
and $s>1$, $\rho_0 > n+2$,
\begin{eqnarray} \label{uniforme}
\text{Tr}[ R_0(z)(f(P)-f(P_0))] &=& \text{Tr}[\w{x}^{-s} (R_0 + O(\mid z\mid^{\epsilon})) \w{x}^{-s} \nonumber \\
&& \ \ \ \ \w{x}^{s} (f(P)-f(P_0)) \w{x}^{s}] \nonumber \\
&=& O(1).
\end{eqnarray}
Then, the assumption (\ref{ass3bis}) is satisfied.

%-------------------------------------------------------------------------------
%                       Singularite de la SSF en 0
%-------------------------------------------------------------------------------

\vspace{0.3cm}\noindent
In order to obtain a Levinson theorem, we have also to prove that
$\xi'(\lambda)$ is integrable on $]0,1]$. We have the following result :

\vspace{0.2cm}\noindent

\begin{theorem}\label{singularitexi}
Under the conditions of Theorem \ref{rtth01}, for some $\epsilon_0>0$,
\begin{equation}\label{singularite}
\xi'(\lambda) = O\ \left(\lambda^{-1+\epsilon_0}\right) \ ,\ \lambda \downarrow 0.
\end{equation}
\end{theorem}

\begin{proof}
We shall prove that
\begin{equation}\label{dualite}
\lambda^{1-\epsilon_0} \ \xi' (\lambda) \in L^{\infty}(]0,1[) = \left( L^1(]0,1[) \right)' ,
\end{equation}
or equivalently,
\begin{equation}\label{continuite}
\exists \ C>0, \ \forall \varphi \in C_0^{\infty}(]0,1[), \ \mid \int_{\bR} \
\lambda^{1-\epsilon_0} \ \xi' (\lambda) \ \varphi(\lambda) \ d\lambda \mid \ \leq \ C  \mid\mid \varphi \mid\mid_1 .
\end{equation}

\vspace{0.5cm}\noindent
Let us consider $g\in C_0^{\infty}(]0,1[; \bR)$ and $f \in C_0^{\infty}(\bR;\bR)$ such that $f \equiv 1$ on $[0,2]$,
so we have $fg=g$. For $\lambda \in ]0,1[$ and $\epsilon>0$, we have :
\begin{eqnarray}
\lefteqn{
\text{Tr}~[R(\lambda +i\epsilon) f(P) - R_0(\lambda +i\epsilon) f(P_0) ]} \label{algebre} \\
& = &
\text{Tr}~[(R(\lambda +i\epsilon) - R_0(\lambda +i\epsilon)) f(P) ]
 + \text{Tr}~[R_0(\lambda +i\epsilon) \ (f(P)- f(P_0)].  \nonumber
\end{eqnarray}
By the definition of the SSF,
\begin{equation}\label{defssf}
\text{Tr}~[R(\lambda +i\epsilon) f(P) - R_0(\lambda +i\epsilon) f(P_0) ] =
- \int_{\bR} \ \xi(s) \ \frac{\partial f}{\partial s} (s,\lambda+i\epsilon) \ ds ,
\end{equation}
where we have set
\begin{equation}
f( s,\lambda+i\epsilon) = \frac{1}{s-\lambda-i\epsilon} \ f(s).
\end{equation}
Thus, the left hand side of (\ref{algebre}) is given by
\begin{equation}
(LHS) =
\int_{\bR} \ \xi(s) \ \frac{1}{ (s-\lambda-i\epsilon)^2} \ f(s) \ ds
- \int_{\bR} \ \xi(s) \ \frac{1}{ (s-\lambda-i\epsilon)} \ f'(s) \ ds.
\end{equation}
Using (\ref{Tz}) and (\ref{uniforme}), the right hand side of (\ref{algebre}) satisfies for a suitable real
constant $C$,
\begin{equation}
(RHS) = - \frac{J_0}{\lambda +i\epsilon} + \frac{C}{(\lambda +i\epsilon) \log (\lambda +i\epsilon)}
+ r(\lambda, \epsilon),
\end{equation}
where
\begin{eqnarray}\label{reste}
r(\lambda, \epsilon) &=& O\ \left(|\lambda +i\epsilon|^{-1 +\epsilon_0}\right) +
\text{Tr}~[R_0(\lambda +i\epsilon) \ (f(P)- f(P_0)] \\
&=& O\ \left(\lambda^{-1 +\epsilon_0}\right) \ ,\  \lambda \downarrow 0.
\end{eqnarray}
Thus, we have obtained that
\begin{equation}\label{synthese}
r(\lambda, \epsilon) =  \frac{J_0}{\lambda +i\epsilon} - \frac{C}{(\lambda +i\epsilon) \log (\lambda +i\epsilon)}
 + a(\lambda, \epsilon) - b(\lambda, \epsilon) ,
 \end{equation}
 where
\begin{eqnarray}
a(\lambda, \epsilon) &=& \int_{\bR} \ \xi(s) \ \frac{1}{ (s-\lambda-i\epsilon)^2} \ f(s) \ ds, \\
b(\lambda, \epsilon) &=& \int_{\bR} \ \xi(s) \ \frac{1}{ s-\lambda-i\epsilon} \ f'(s) \ ds .
\end{eqnarray}
We shall multiply (\ref{synthese}) by $g(\lambda)$ and we shall integrate over $\lambda$. First, we remark that
\begin{eqnarray}
\int_{\bR} a(\lambda, \epsilon) \ g(\lambda) \ d\lambda &=&
-\int_{\bR} \ \xi(s) f(s) \left( \int_{\bR} \ \frac{\partial}{\partial \lambda}
\left(\frac{1}{\lambda-s +i\epsilon} \right) \ g(\lambda) \ d\lambda \right) ds  \nonumber \\
&=&
\int_{\bR} \ \xi(s) f(s) \left( \int_{\bR} \
\frac{1}{\lambda-s +i\epsilon}  \ g'(\lambda) \ d\lambda \right) ds.
\end{eqnarray}
Thus,
\begin{eqnarray} \label{thelast}
\int_{\bR} r(\lambda, \epsilon) \ g(\lambda) \ d\lambda &=&
 \int_{\bR} \frac{J_0}{\lambda+i\epsilon} \ g(\lambda) \ d\lambda -
\int_{\bR} \frac{C}{(\lambda +i\epsilon) \log (\lambda +i\epsilon)} \ g(\lambda) \ d\lambda \nonumber \\
& & + \int_{\bR} \ \xi(s) f(s) \left( \int_{\bR} \
\frac{1}{\lambda-s +i\epsilon}  \ g'(\lambda) \ d\lambda \right) ds \nonumber\\
& & + \int_{\bR} \ \xi(s) f'(s) \left( \int_{\bR} \
\frac{1}{\lambda-s +i\epsilon}  \ g(\lambda) \ d\lambda \right) ds.
\end{eqnarray}
Taking the imaginary part in (\ref{thelast}) and using that
\begin{equation}
\lim_{\epsilon \downarrow 0}  \ {\rm{Im}}\  \frac{1}{\lambda-s+i\epsilon} =\pi \ \delta_{\lambda=s},
\end{equation}
we obtain easily
\begin{equation}
 \lim_{\epsilon \downarrow 0} {\rm{Im}}\  \int_{\bR} r(\lambda, \epsilon) \ g(\lambda) \ d\lambda =
{\pi} \  \int_{\bR} \ \xi(s) \ (fg)'(s) \ ds.
\end{equation}
Now, we remark that :
\begin{equation}
\int_{\bR} \ \xi(s) \ (fg)'(s) \ ds  = \int_{\bR} \ \xi(s) \ g'(s) \ ds = - \int_{\bR} \ \xi'(s) \ g(s) \ ds ,
\end{equation}
since $\xi \in C^{\infty}(]0, +\infty[)$. Then, we choose $g(s) = s^{1-\epsilon_0} \varphi(s)$ with $\varphi \in
C_0^{\infty}(]0,1[)$, and using (\ref{reste}), we have proved (\ref{continuite}).
\end{proof}

%-------------------------------------------------------------
%       ENONCE THEOREME DE LEVINSON
%-------------------------------------------------------------

\vspace{0.5cm}\noindent
Now, we are able to prove a Levinson's theorem for critical potentials. To do it,   we recall that
if $v$ satisfies (\ref{eq1.2}), one has the high energy asymptotics, (\cite{Robert01}, Theorem 1.2) :

\begin{equation}\label{asymptSSF}
\xi'(\lambda) \sim  \sum_{j \ge 1} c_j \ \lambda^{\frac{n}{2}-j-1} \ ,\ \lambda \rightarrow +\infty.
\end{equation}
This implies in particular that the condition (\ref{ass5}) is satisfied. We also need the following asymptotics which comes from the functional calculus on pseudodifferential operators
(see \cite{Robert01}, Theorem 1.1): Let $\chi  \in {\mathcal S}(\bR)$, with $\chi \equiv 1$ in a neighborhood
of $0$. We have :

\begin{equation}\label{calculfonc}
\text{Tr}~ [ \chi (\frac{P}{R}) - \chi (\frac{P_0}{R}) ] \sim \sum_{j\geq 1} \ \beta_j \ R^{\frac{n}{2}-j} \ ,\
R\rightarrow +\infty,
\end{equation}
where the coefficients $\beta_j$ are distributions on $\chi$ and is calculable, in principle,
in terms of the symbols of $P$ and $P_0$.

\vspace{0.5cm}\noindent
The main result of this section is the following :

\vspace{0.2cm}\noindent

\begin{theorem}\label{Levinson}
Assume $\rho_0>max( 6,n+2)$. Then, we have :
\begin{equation}
\int_0^{\infty} (\xi'(\lambda) - \sum_{j=1}^{[\frac{n}{2}]} c_j
\ \lambda^{[\frac{n}{2}]-1-j} ) \ d \lambda = -( \mathcal{N}_- + J_0) + \beta_{n/2},
\end{equation}
where $\beta _{n/2}$  depends only on $n$, $v$ and $V$. If $n$ is odd, $\beta _{n/2} =0$. If $n$ is even,
$c_{n/2} =0$.
\end{theorem}

\begin{proof}
Let $\chi\in C_0^\infty(\bR)$ such that $\chi(\lambda)\equiv 1$ in a neighborhood of $0$.
We use Theorem \ref{rfth01} with $f(\lambda) = \chi(\frac{\lambda}{R})$. The proof is divided in two steps.

\pagebreak

\noindent {\it{Case 1. The dimension $n$ is odd.}}

\vspace{0.2cm}\noindent
We first compute the second term on the left hand side of
(\ref{rffl02}) as follows :
\begin{eqnarray*}
\int_0^\infty \chi(\frac{\lambda}{R}) \ \xi' (\lambda)~d\lambda &=&
\int_0^\infty \chi(\frac{\lambda}{R}) \ [\xi'(\lambda)-\sum_{j=1}^{[\frac{n}{2}]} c_j \  \lambda^{\frac{n}{2}
-1-j}]~d\lambda +\sum_{j=1}^{[\frac{n}{2}]} c_j  \int_0^\infty \chi(\frac{\lambda}{R}) \
\lambda^{\frac{n}{2} -1-j}~d\lambda \\
&=& \int_0^\infty \chi(\frac{\lambda}{R})\ [\xi'(\lambda)-\sum_{j=1}^{[\frac{n}{2}]} c_j \ \lambda^{\frac{n}{2}
-1-j}]~d\lambda + \sum_{j=1}^{[\frac{n}{2}]} d_j \ R^{\frac{n}{2}-j},
\end{eqnarray*}
with ${\displaystyle{d_ j= c_j \ \int_0^\infty \chi(t) \ t^{\frac{n}{2}-1-j} ~dt}}$.
Using (\ref{rffl02}) and (\ref{calculfonc}), one has
\begin{equation}\label{ltfl04}
\int_0^\infty \chi(\frac{\lambda}{R})\
[\xi' (\lambda)-\sum_{j=1}^{[\frac{n}{2}]} c_j \ \lambda^{\frac{n}{2} -1-j}]~d\lambda
= -(\mathcal{N}_- +J_0) + \sum_{j=1}^{[\frac{n}{2}]}
(\beta_j-d_j) \ R^{\frac{n}{2}-j} +O(R^{-\epsilon})
\end{equation}
where ${\displaystyle{\epsilon  = 1 -\frac{n}{2}+[\frac{n}{2}]>0}}$.

\vspace{0.2cm}\noindent
Now, let us study the left hand side of (\ref{ltfl04}). We have :
\begin{eqnarray*}
\lefteqn{\int_0^\infty \chi(\frac{\lambda}{R})\ [\xi'(\lambda)-\sum_{j=1}^{[\frac{n}{2}]} c_j\ \lambda^{\frac{n}{2} -1-j}]
~d\lambda}\\
 &=&
\int_0^1 \chi(\frac{\lambda}{R})\
[\xi'(\lambda)-\sum_{j=1}^{[\frac{n}{2}]} c_j \ \lambda^{\frac{n}{2} -1-j}]~d\lambda  + \int_1^\infty \chi(\frac{\lambda}{R})\
[\xi'(\lambda)-\sum_{j=1}^{[\frac{n}{2}]} c_j \ \lambda^{\frac{n}{2}-1-j}]~d\lambda.
\end{eqnarray*}

\vspace{0.1cm}\noindent
By Theorem \ref{singularitexi}, $\xi'(\lambda)$ is integrable on $]0,1[$, thus
\begin{equation}
\lim_{R\rightarrow +\infty}\ \int_0^1 \chi(\frac{\lambda}{R})\
[\xi'(\lambda)-\sum_{j=1}^{[\frac{n}{2}]} c_j \ \lambda^{\frac{n}{2}
-1-j}]~d\lambda = \int_0^1 [\xi'(\lambda)-\sum_{j=1}^{[\frac{n}{2}]}
c_j \ \lambda^{\frac{n}{2} -1-j}]~d\lambda.
\end{equation}

\noindent
In the same way, (\ref{asymptSSF}) implies
\begin{equation}
|\xi'(\lambda)-\sum_{j=1}^{[\frac{n}{2}]} c_j \ \lambda^{\frac{n}{2}-1-j}|
\le C \lambda^{-\nu} ,
\end{equation}
 with $\nu =2 -\frac{n}{2}+[\frac{n}{2}] >1$. Therefore,
\begin{eqnarray*}
\lim_{R\rightarrow +\infty}\ \int_1^\infty \chi(\frac{\lambda}{R})\
[\xi'(\lambda)-\sum_{j=1}^{[\frac{n}{2}]} c_j \ \lambda^{\frac{n}{2}
-1-j}]~d\lambda = \int_1^\infty [\xi' (\lambda)-\sum_{j=1}^{[\frac{n}{2}]} c_j \ \lambda^{\frac{n}{2}
-1-j}]~d\lambda.
\end{eqnarray*}

\vspace{0.3cm}\noindent
It follows that
\begin{equation}
\lim_{R\rightarrow +\infty}\ \int_0^\infty \chi(\frac{\lambda}{R})\
[\xi'(\lambda)-\sum_{j=1}^{[\frac{n}{2}]} c_j\ \lambda^{\frac{n}{2} -1-j}]~d\lambda
= \int_0^\infty [\xi'(\lambda)-\sum_{j=1}^{[\frac{n}{2}]} c_j \ \lambda^{\frac{n}{2} -1-j}]~d\lambda .
\end{equation}
Taking $R\rightarrow +\infty$ in both sides of (\ref{ltfl04}), we deduce
that $\beta_j = d_j$. Thus, we get the Levinson's Theorem for odd dimension :
\begin{equation}
\int_0^{\infty} (\xi'(\lambda) - \sum_{j=1}^{[\frac{n}{2}]} c_j
\ \lambda^{[\frac{n}{2}]-1-j} ) \ d \lambda = -( \mathcal{N}_- + J_0).
\end{equation}

\vspace{0.5cm}\noindent
{\it{Case 2: The dimension $n$ is even.}}

\vspace{0.3cm}\noindent
First, let us study the case $n=2$. Using Theorem \ref{rfth01}, we have as previously :
\begin{eqnarray}\label{negal2}
\lefteqn{\int_0^1\chi(\frac{\lambda}{R}) \ \xi'(\lambda) ~d\lambda
+ \int_1^\infty \chi(\frac{\lambda}{R})\  [\xi'(\lambda)- \frac{c_1}{\lambda}]~d\lambda}  \nonumber \\
&=& -( \mathcal{N}_- + J_0) + \beta_1 + O(\frac{1}{R}) - c_1 \int_1^{+\infty} \chi(\frac{\lambda}{R})
\frac{d\lambda}{\lambda}.
\end{eqnarray}
We take $R \rightarrow + \infty$ in (\ref{negal2}), and remarking that
\begin{equation}
\int_1^\infty\chi(\frac{\lambda}{R}) ~\frac{d\lambda}{\lambda}  \sim \log R,
\end{equation}
we deduce $c_1=0$ and we have obtained the Levinson theorem in dimension $n=2$ :
\begin{equation}
\int_0^{+\infty} \xi'(\lambda) ~d\lambda = -( \mathcal{N}_- + J_0) + \beta_1 .
\end{equation}

\vspace{0.3cm}\noindent
Now, assume $n \geq 4$. We write $n=2p$ with $p \geq 2$. We define for $1 \leq j\leq  p-1$,
\begin{equation}
d_j  = c_j \int_0^\infty \chi(t) \ t^{p-1-j} ~dt .
\end{equation}
As for the case $n=2$, we obtain easily :
\begin{eqnarray}\label{npaire}
\lefteqn{\int_0^1\chi(\frac{\lambda}{R})\  [\xi'(\lambda)-\sum_{j=1}^{p-1} c_j\ \lambda^{p -1-j}]~d\lambda
+ \int_1^\infty \chi(\frac{\lambda}{R})\  [\xi'(\lambda)-\sum_{j=1}^{p} c_j\ \lambda^{p -1-j}]~d\lambda}
\nonumber \\
&=& -(\vN_- +J_0) +\sum_{j=1}^{p-1} (\beta_j-d_j)\  R^{p-j} +\beta_p +O(\frac{1}{R})
- c_p \int_1^\infty \chi(\frac{\lambda}{R}) ~\frac{d\lambda} {\lambda}.
\end{eqnarray}
As previously, we take $R \rightarrow +\infty$ in (\ref{npaire}) and we deduce that $\beta_j= d_j$ and $c_p =0$.
Therefore, we get the Levinson's theorem for even dimension:
\begin{equation}
\int_0^\infty (\xi'(\lambda) -\sum_{j=1}^{p-1} c_j \ \lambda^{p-1-j})~d\lambda=-(\mathcal{N}_-+J_0)+\beta_p.
\end{equation}
\end{proof}

\begin{remark}
Of course, the values of $\beta_p$ appearing in the Levison theorem are independent of the cutoff function
$\chi$. We can use the functional calculus on pseudodifferential operators (\cite{HeRo})
 with the Hamiltonians $P_0 = -\Delta +v(x)-V(x)$ and $P= -\Delta + v(x)$, to compute $\beta_p$.

\vspace{0.3cm}\noindent
If $f \in C_0^{\infty}(\mathbb{R})$,  we recall that $f(tP)-f(tP_0)$ is a $\sqrt{t}$-admissible operator, i.e
\begin{equation}
f(tP)-f(tP_0) = Op_{\sqrt{t}} \left( a_f (t) \right),
\end{equation}
with
\begin{equation}
a_f (t) \sim \sum_{j \geq 1} a_{f,j} \ t^{\frac{j}{2}}.
\end{equation}
The symbols $a_{f,j}$ are defined as
\begin{equation}
a_{f,j} (x, \xi)  = \sum_{k=1}^{2j-1} (-1)^k d_{j,k} (x, \xi) \ f^{(k)} (\xi^2),
\end{equation}
where the $d_{j,k}$ are universal polynomial functions on $\xi$ which do not depend on $f$. For example, one has
\begin{eqnarray}
d_{2,1} (x, \xi) &=& V(x), \ d_{2,2} (x, \xi)= 0, \ d_{2,3} (x, \xi)=0. \\
d_{4,1} (x, \xi) &=& 0, \ d_{4,2} (x, \xi) = 2v(x)V(x) - V^2 (x), \\
d_{4,3} (x, \xi) &=& \sum_i \partial_i^2 V(x) \ \xi_i^2 -2 \sum_{i<j} \partial_{ij} V(x) \ \xi_i \xi_j, \\
d_{4,k} (x, \xi) &=0& \ if\  4 \leq k\leq 7.
\end{eqnarray}

\vspace{0.5cm}\noindent
When $t \rightarrow 0$,
\begin{equation}
\text{Tr}~ [f(tP)-f(tP_0)] \sim t^{-\frac{n}{2}} \sum_{j\geq 1} \beta_j t^j,
\end{equation}
with
\begin{equation}
\beta_j = (2\pi)^{-n} \sum_{k=1}^{4j-1} \ \int_{\mathbb{R}^n} \int_{\mathbb{R}^n}  d_{2j,k}(x,\xi) \ f^{(k)}(\xi^2) \ dx \ d\xi.
\end{equation}

\vspace{0.3cm}\noindent
As a consequence, if we set,
\begin{equation}
\gamma_n \ = Vol \ (S^{n-1}) = \frac {2 \pi^{n/2}}{\Gamma(n/2)} ,
%\gamma_n \ =\ \frac{\frac{n}{2}-1}{(2\pi)^{\frac{n}{2}} \ \Gamma(\frac{n}{2})} \ ,
\end{equation}
where $\Gamma$ is the well-known Gamma function, we can show :

\vspace{0.3cm}\noindent

\begin{itemize}

\item In dimension $n=2$ :
\begin{equation}
\beta_1 = \frac{1}{(2\pi)^2}\ \frac{\gamma_2}{2}\   \int_{\mathbb R^2} V(x) ~dx .
\end{equation}

\item In dimension $n=4$ :
\begin{equation}
\beta_2 = \frac{1}{(2\pi)^4}\ \frac{\gamma_4}{2}\   \int_{\mathbb R^4}
( 2 v(x) V(x) - V^2 (x) )~dx .
\end{equation}

\end{itemize}
\end{remark}

% -----------------------------   Prudence ----------------------

%\item In dimension $n=6$ :
%$$
%\beta_3 =   \frac{1}{(2\pi)^6}\ \frac{\gamma_6}{6}\
%\int_{\mathbb R^6} (3 v^2 (x) V(x) - 3 v(x) V^2 (x)  + V^3 (x)
%$$
%\begin{equation}
%\hspace{2.6cm} -\mid \nabla V (x)\mid^2 +2 \nabla V(x)\cdot\nabla v(x)) ~dx .
%\end{equation}

%\end{itemize}
%\end{remark}
%\bigskip

%The above results can also be deduced by applying Corollary 3.20 of \cite{Dauge01} to
%\[
%Tr (f(tP^R) - f(tP_0^R)) = Tr (f(tP^R) - f (-t\Delta)) -  Tr (f(tP_0^R)) -f(-t\Delta))
%\]
%where $P^R= - \Delta + v(x) \theta (x/R)$ and $P_0^R = -\Delta + (v(x) -V(x)) \theta(x/R)$  and $\theta
%\in C_0^\infty(\bR^n)$  with $\theta (x) =1$ if $|x| \le 1$, and then taking the limit $R \to \infty$.

% ----------------------------------------------

%\bibliography{polymeres,controlabilite}
%\addstarredchapter{Bibliographie}

\end{document}